\documentclass[a4paper,11pt]{article}
\usepackage{amsfonts,amsmath,amssymb,amsthm}
\usepackage{mathrsfs}
\usepackage{ucs} 
\usepackage[utf8x]{inputenc}
\usepackage[T1]{fontenc}
\usepackage[english]{babel}
\usepackage[all]{xy}
\usepackage{verbatim}
\usepackage{bbm}
\usepackage{hyperref}
\hypersetup{colorlinks,%
citecolor=red,%
filecolor=black,%
linkcolor=red,%
urlcolor=blue,%
pdftex}
\usepackage{makeidx}

\newcommand{\an}{\mathrm{an}}
\newcommand{\zar}{\mathrm{zar}}
\newcommand{\rig}{\mathrm{rig}}
\newcommand{\dR}{\mathrm{dR}}
\newcommand{\Q}{\mathbb{Q}}
\newcommand{\C}{\mathbb{C}}
\newcommand{\Z}{\mathbb{Z}}
\newcommand{\R}{\mathbb{R}}

\newcommand{\A}{\mathbb{A}}
\newcommand{\K}{\mathbb{K}}

\newcommand{\Sch}{\mathrm{Sch}}
\newcommand{\iHom}{{\mathcal{H}om}}
\newcommand{\scrX}{\mathscr{X}}
\newcommand{\scrY}{\mathscr{Y}}
\newcommand{\scrV}{\mathscr{V}}
\newcommand{\scrW}{\mathscr{W}}
\newcommand{\scrZ}{\mathscr{Z}}
\newcommand{\scrD}{\mathscr{D}}

\renewcommand{\O}{\mathcal{O}}

\DeclareMathOperator{\Cone}{Cone}

\DeclareMathOperator{\Hom}{Hom}

\DeclareMathOperator{\id}{id}
\DeclareMathOperator{\Spec}{Spec}

\DeclareMathOperator{\tr}{tr}

\DeclareMathOperator{\gr}{gr}

\DeclareMathOperator{\Coim}{Coim}

\DeclareMathOperator{\Ker}{Ker}
\DeclareMathOperator{\Coker}{Coker}

\DeclareMathOperator{\supp}{supp}
\DeclareMathOperator*{\colim}{colim}
\renewcommand{\Im}{\mathrm{Im}}
\DeclareMathOperator{\length}{length}

\newcommand{\vuoto}{\varnothing}
\newcommand{\x}{\times}
\newcommand{\ox}{\otimes}
\newcommand{\iso}{\cong}

\newcommand{\inv}{^{-1}}
\newcommand{\dual}{^{\vee}}

\usepackage{amscd}

\usepackage{txfonts}
\usepackage{enumerate}
\usepackage{fullpage}

\newtheorem{thr}{Theorem}[section]
\newtheorem{lmm}[thr]{Lemma}
\newtheorem{prp}[thr]{Proposition}
\newtheorem{crl}[thr]{Corollary}

\theoremstyle{definition}\newtheorem{dfn}[thr]{Definition}
\theoremstyle{remark}\newtheorem{rmk}[thr]{Remark}
\theoremstyle{remark}\newtheorem{exm}[thr]{Example}
\theoremstyle{remark}\newtheorem{prg}[thr]{-}

\newcommand{\syn}{{\rm syn}}
\newcommand{\reg}{{\rm reg}}
\newcommand{\Gdm}{{\rm Gd}}
\newcommand{\abs}{{\rm abs}}
\title{Cycle classes and the 
syntomic regulator}
\author{B. Chiarellotto, A. Ciccioni, N. Mazzari}
\date{}
\makeindex

\begin{document}
\maketitle
\abstract{Let $\scrV=\Spec(R)$ and  $R$ be a complete discrete valuation ring  of mixed characteristic  $(0,p)$.
For any  flat  $R$-scheme $\scrX$  we prove the compatibility of the de Rham fundamental  class of the generic fiber and the rigid fundamental class of the special fiber.  We use this result to construct a syntomic regulator map $\reg_\syn:CH^i(\scrX/\scrV,2i-n)\to H^n_{\syn}(\scrX,i)$, when  $\scrX$ is smooth over $R$, with values on the  syntomic cohomology defined  by  A. Besser. Motivated by the previous result  we also prove some of the Bloch-Ogus axioms for the syntomic  cohomology theory, but viewed as an absolute cohomology theory.}

%
%
%
\section*{Introduction}
Let $\scrV=\Spec(R)$, with $R$ a complete discrete valuation ring of mixed characteristic and perfect residue field. Given  $\scrX$  an algebraic  $\scrV$-scheme one can consider the de Rham cohomology of its generic fiber $\scrX_K$ and the rigid cohomology of its special fiber $\scrX_k$. These two cohomology groups are related (see \cite[\S6]{BalCaiFio:Poi04}) by a canonical \emph{cospecialization} map $cosp:H_{\rig,c}^n(\scrX_k)\to H^n_{\dR,c}(\scrX_K)$ (not an isomorphism in general). There is also a notion of rigid and de Rham cycle class. The starting result of this paper is the compatibility of this cycle classes w.r.t. the cospecialization map (see \ref{thr:key} for the precise statement). \\
In the case $\scrX$ is smooth (possibly non-proper) over $\scrV$ we get the following corollary  (see Cor.~\ref{crl:cycle class + sp}): let $sp_{CH}:CH^*(\scrX_K)\ox\Q\rightarrow CH^*(\scrX_k)\ox\Q$ be the specialization of Chow rings constructed by Grothendieck in \cite[Appendix]{SGA6}, then the following diagram is  commutative
\begin{equation*}
\xymatrix{
CH^q (\scrX_K/K)\ox\Q\ar[d]_{sp_{CH}}\ar[r]^{\eta_\dR}& H^{2q}_\dR(\scrX_K)\ar[d]^{sp}\\
CH^q(\scrX_k/k)\ox\Q\ar[r]_{\eta_\rig} & H^{2q}_\rig(\scrX_k/K) &,}
\end{equation*}
where $\eta_\dR$ (resp. $\eta_\rig$) is the de Rham (resp. rigid) cycle class map and $sp$ is the Poincaré dual of $cosp$.\\
In the proof we use  the main results  of \cite{BalCaiFio:Poi04}, \cite{BosLutRay:For95} and \cite{Pet:Cla03}. 

This result can be viewed as a generalization of a theorem of Messing \cite[Theorem~B3.1]{GilMes:Cyc87} where he further assumes $\scrX$ to be proper  (not only smooth) over $\scrV$. In that case rigid cohomology coincides with the  crystalline one and the map $sp$ is an isomorphism (\cite{Ber:Fin97}).

This compatibility result is the motivation for  an alternative construction of the regulator map (see Prop.~\ref{prp:regsyn})
	\[
		\reg_\syn:CH^{i}(\scrX/\scrV,2i-n)\to H^n_{\syn}(\scrX,i)
	\]
with values in the syntomic cohomology group defined by Besser in \cite{Bes:Syn00} (for $\scrX$ smooth over $\scrV$). For this proof we use  an argument of Bloch \cite{Blo:Alg86h} and the existence of  a syntomic cycle class (see Prop.~\ref{prp:syncyc}). 

The aforementioned results motivated us to investigate further the properties of the syntomic cohomology. 
We are not able to formulate even the basic Bloch-Ogus axioms using  Besser's framework. Thus
we  followed   Bannai's interpretation of syntomic cohomology as an absolute one \cite{Ban:Syn02}. To this aim we define a triangulated  category of $p$-adic Hodge complexes, $pHD$ (see Def.~\ref{def:pHC}). An object $M$ of $pHD$  can be represented by a diagram of the form 
$M_\rig\to M_K\gets M_\dR$ where $M_?$ is a complex of $K$-vector spaces endowed with a Frobenius automorphism when $?=\rig$ and with a filtration when $?=\dR$. On $pHD$ there is a naturally defined  tensor product and we denote the unit object of $pHD$ by $\K$. The main difference  w.r.t. \cite{Ban:Syn02}  being that the maps in the diagram are not necessarily quasi-isomorphisms. \\
From \cite{Bes:Syn00} we get (in Prop.~\ref{prp:RGamma functors}) that there are functorial $p$-adic Hodge complexes $R\Gamma(\scrX)$ 
\[
	R\Gamma_\rig(\scrX)\to R\Gamma_K(\scrX)\gets R\Gamma_\dR(\scrX)
\]
 inducing the specialization map in cohomology (\textit{i.e.}, taking the cohomology of each element of the diagram).  Meanwhile  we show how the constructions made by Besser may be obtained using the theory of generalized Godement resolution (also 
called the bar resolution). In particular we use the results of \cite{PutSch:Poi95} in order to have enough points for rigid analytic spaces.
 Further we   consider the twisted version  $R\Gamma(\scrX)(i)$, which is given by the same complexes, but with the Frobenius (resp. the filtration)  twisted by $i$ (see Remark~\ref{rmk:tatetwist}).

Then we prove (see Prop.~\ref{prp:bessercomp}) that  the syntomic cohomology groups $H^n_\syn (\scrX,i)$ of \cite{Bes:Syn00} are isomorphic  to  the (absolute cohomology) groups  $H^n_\abs (\scrX,i):=\Hom_{pHD}(\K,R\Gamma(\scrX)(i)[n])$. This result generalizes  that of \cite{Ban:Syn02} (which was given for some particular $\scrV$-schemes with good compactification) to any smooth algebraic $\scrV$-scheme.

For this absolute cohomology we can prove some of the Bloch-Ogus axioms. \\
In fact we construct a $p$-adic Hodge complex $R\Gamma_c(\scrX)(i)$ related to rigid and de Rham cohomology with compact support. So we can define an absolute cohomology with compact support functorial with respect to proper maps
$H^n_{\abs,c} (\scrX,i):=\Hom_{pHD}(\K,R\Gamma_c(\scrX)(i)[n])$
and an absolute homology theory $H_{n}^\abs(\scrX,i):=\Hom_{pHD}(R\Gamma_c(\scrX)(i)[n],\K)$.

We want to point out that the above constructions are essentially consequences of the work done by Besser and Bannai, but it seems hard to prove the following results without the formalism of Godement resolutions we develop in \S~\ref{sec:gdm}: let $\scrX$ be a smooth scheme over $\scrV$, then
\begin{enumerate}[(i)]
	\item there is a cup product pairing 
	\[
		H^n_\abs(\scrX,i)\ox H^m_{\abs,c}(\scrX,j)\to H^{n+m}_{\abs,c}(\scrX,i+j)
	\]
	induced by the natural pairings defined on the cohomology of the generic and the special fiber (see Cor.~\ref{crl:pairing});
	\item there is a Poincaré duality isomorphism (see Prop.~\ref{prp:dual}); 
	\item there is a Gysin map, \textit{i.e.} given a proper morphism $f:\scrX\to \scrY$  of smooth algebraic $\scrV$-schemes of relative dimension  $d$ and $e$, respectively, then there is a canonical map
	\[
		f_*:H_\abs^n(\scrX,i)\rightarrow H_\abs^{n+2c}(\scrY,i+c) 
	\]
	where $c=e-d$ (see Cor.~\ref{crl:gysin}).
\end{enumerate}
 %
%
%
%
%
\paragraph*{Notation}
In this paper $R$ is a complete discrete valuation ring with fraction field $K$ and residue field $k$, with $k$ perfect. We assume $char(K)=0$ and $char (k)=p>0$. 

We denote by $R_0$ the ring of Witt vectors of $k$  and $K_0$ is its field of fractions. The Frobenius of $K_0$ is denoted by $\sigma$. \\
The category of bounded complexes of $K$-vector spaces is denoted by $C^b(K)$.

If $V$ is a $K$-vector space, then $V\dual$ is the dual vector space. 

We use $X,Y,...$ for schemes over $k$ or $K$; $\cal X, \cal Y, ...$ for $K$-analytic spaces; 
$\sf P, Q$ for formal $K$-schemes and ${\sf P}_K, {\sf Q}_K$ for the associated Raynaud fibers; $\scrX, \scrY$ 
for (algebraic) $\scrV$-schemes, where $\scrV=\Spec(R)$. Finally, the $p$-adic completion of a $\scrV$-scheme $\scrX$ will be denoted by $\widehat{\scrX}$.
\section{Cycle classes}
\begin{prg}[Higher Chow groups]
	Following Bloch \cite{Blo:Alg86h} we recall the definition of the higher Chow groups $CH^*(X/K,*)$ of a $K$-scheme $X$.\\
	For any $n\le 0$ let $\Delta^n=\Spec(\Z[t_0,...,t_n]/(\sum_it_i-1))$ with face maps $\partial_i(n):\Delta^n\to \Delta^{n+1}$, which in coordinates are given by $\partial_i(n)(t_0,...,t_n)=(t_0,..,t_{i-1},0,t_{i+1},..,t_n)$. Let $X$ be a smooth $K$-scheme of relative dimension $d$ (this hypothesis is not necessary, but we deal only with smooth schemes). Let $z^q(X /K,0)$ be the free abelian group generated by the irreducible and closed sub-schemes of $X$ of codimension $q$. We denote by $z^q(X / K,n)$ the free abelian group generated by the closed sub-schemes $W \subset \Delta^n_{X}:=X\x \Delta^n$ such that $W\in z^q(\Delta^n_X /K,0)$ and meets all faces properly: \textit{i.e.}, if $F\subset \Delta^n_X$ is a face of codimension $c$, then the codimension of the irreducible components of the intersection $F\cap W$ is 
greater than or equal to $c+q$ on $\Delta^n_X$.

	 Using the differential $\sum_i(-1)^i\partial_i^*(n):z^q(X/K,n+1)\to z^q(X/K,n)$ one obtains a complex of abelian groups $z^q(X/K,*)$. We set $\Gamma_X^i(q):=z^q(X/K,2q-i)$ and 
	\[
		CH^{q}(X/K,2q-i):=H^i(\Gamma_X(q))\ .
	\] 
	these groups are in fact isomorphic to the  Voevodsky--Suslin  motivic cohomology $H_{mot}^i(X/K,q)$ of 
the generic fiber $X$ (\cite[Theorem 19.1]{MazVoeWei:Lec06}).
\end{prg}
\begin{rmk}[Relative cycles]
	Let $\scrX$ be an algebraic and flat $\scrV$-scheme. By the theory of relative cycles \cite{SusVoe:Rel00} one can define the group $z^q(\scrX/\scrV,0)$ to be the free abelian group generated by \emph{universally integral relative cycles} of codimension $q$ (we can use the codimension because $\scrX$ is assumed to be equidimensional over $\scrV$). 
	By \cite[Part I, Lemma~1.2.6]{Ivo:Rea05} $z^q(\scrX/\scrV,0)$ is the free abelian group generated by   the closed sub-schemes $\scrW\subset \scrX$ which are integral, of codimension $q$ and flat over $\scrV$. Then we can define the groups $z^q(\scrX/\scrV,n)$ as follows:  let $z^q(\scrX/\scrV,n)$ be the free abelian group generated by  the integral and flat $\scrV$-schemes $\scrW\in z^q(\Delta^n_\scrX,0)$ meeting all faces properly and such that $\partial_i(n-1)^*\scrW$ is flat over $\scrV$ for all $i$. Thus we can form a complex $z^q(\scrX/\scrV,*)$  using the same boundary maps of $z^q(X/K,*)$.
\end{rmk}
\begin{dfn}
	With the above notation we define the higher Chow groups of $\scrX$ over $\scrV$ to be
	\[
		CH^{q}(\scrX/\scrV,2q-i):=H^i(\Gamma_X(q))
	\]
	where $\Gamma_{\scrX/\scrV}^i(q):=z^q(\scrX/\scrV,2q-i)$.
\end{dfn}
\begin{rmk}\label{rmk:gamma}
	\begin{enumerate}[(i)]
		\item Recall that by Lemma~\ref{lmm:platification} any closed and flat sub-scheme of $\scrX$ is completely determined by its generic fiber. Then $z^q(\scrX/\scrV,*)$ is a sub-complex of $z^q(\scrX_K/K,*)$ inducing a canonical map in (co)-homology
		\[
			\gamma:CH^{q}(\scrX/\scrV,2q-i):=H^i(\Gamma_{\scrX/\scrV}(q))\longrightarrow CH^{q}(\scrX_K/K,2q-i)\ .
		\]
		\item 	It follows easily by the Snake lemma that, for  $i=2q$, the map  $\gamma:CH^{q}(\scrX/\scrV,0)\to CH^{q}(\scrX_K/K,0)$  is  surjective.	In the general case with don't know wether this map is injective or surjective.
	\end{enumerate}
\end{rmk}
%
%
\begin{prg}[De Rham and rigid fundamental/cycle classes]
	In the following  we refer to \cite{Har:On-75,BalCaiFio:Poi04,Pet:Cla03} for the definitions and the properties of the (algebraic) de Rham and the rigid cohomology theory.\\
	Let $W$ be an integral scheme of dimension $r$ over $K$ (resp. over $k$). Then we can associate to $W$ its de Rham  (resp. rigid) fundamental class which is an element of  the dual of the top de Rham (resp. rigid) cohomology with compact support
	\[
		[W]_\dR=\tr_\dR\in H^{2r}_{\dR,c}(W)\dual\ ,\ \text{resp.}\ [W]_\rig=\tr_\rig\in H^{2r}_{\rig,c}(W)\dual \ .
	\]
	For the de Rham cohomology this class is first defined in \cite[\S7]{Har:On-75} as an element of the de Rham homology; by Poincaré duality \cite[Theorem 3.4]{BalCaiFio:Poi04} it corresponds to the trace map. 
	The rigid case is treated in \cite[\\SS2.1, 6]{Pet:Cla03}. 
	
	Now let $X$ be a $K$-scheme  (resp. $k$-scheme) of dimension $d$ and $w=\sum_i n_i W_i\in z^q(X/K,0)$ (resp. $\in z^q(X/k,0)$)  be a dimension $r$ cycle on $X$. The cohomology with compact support is functorial with respect to proper maps, hence there is a canonical map
	\[
		f:\oplus_i H^{2r}_{?,c}(W_i)\dual\to H^{2r}_{?,c}(|w|)\dual
	\] 
	where $|w|=\cup_iW_i$ is the support of $w$ and $?=\dR,\rig$ accordingly to the choice of the base field. With the above notations we define the de Rham (resp. rigid) cycle class of $w$ as
	\[
		[w]_?:=f\left(\sum_i[W_i]_?	\right)\in H^{2r}_{?,c}(|w|)\dual\qquad ?=\dR, \rig\ .
	\] 
	Again by functoriality this defines an element of $H^{2r}_{?,c}(X)\dual$.
	\end{prg}
Let $\scrX$ be a flat $\scrV$-scheme  and let $w\in z^q(\scrX/\scrV,0)$ be a relative cycle. We can write $w=\sum_i n_i \scrW_i$, where $\scrW_i$ is an integral flat $\scrV$-scheme closed in $\scrX$ and of codimension $q$. Then $w$ defines a cycle $w_K\in z^q(\scrX_K/K,0)$ (resp. $w_k\in z^q(\scrX_k/k,0)$): on the generic fiber we get simply  $w_K=\sum_i n_i (\scrW_i)_K$.  However, on the special fiber 
we must write the irreducible decomposition  of each $(\scrW_i)_k^{\rm red}$, say $W_{i,1}\cup \cdots \cup W_{i,r_i}$, 
and then consider the multiplicities, \textit{i.e.},
\[
	w_k=\sum_i\left(n_i\sum_jm_{i,j}W_{i,j}\right)
\] 
where $m_{i,j}:=\length \O_{\scrW_k, W_{i,j}}$.

We can consider  the de Rham and the rigid  cycle classes of $w$, \textit{i.e.}, 
\[
	[w_K]_\dR\in H^{2r}_{\dR,c}(|w_K|)\dual\ ,\quad [w_k]_\rig\in H^{2r}_{\rig,c}(|w_k|)\dual \ .
\] 
Note that the the  rigid homology groups $H_{2r}^{\rig}(|w_k|)$ are defined as the dual  of $H^{2r}_{\rig,c}(|w_k|)$ (\textit{cf.} \cite[\S2]{Pet:Cla03}).
We are going to prove that these cycle classes are compatible under (co)-specialization and that they induce a well defined syntomic cohomology class.
\begin{thr}[Cycle class compatibility]\label{thr:key}
	Let $\scrX$ be a flat $\scrV$-scheme of relative dimension $d$ and let $w \in z^q(\scrX/\scrV,0)$  be a relative cycle of codimension $q$ (and relative dimension $r=d-q$). Then $cosp([w_k]_\rig)=[w_K]_\dR$.
\end{thr}
\begin{proof}
	First of all consider the basic case: $w=\scrW$ is an integral closed sub-scheme of $\scrX$, smooth over $\scrV$.
	By \cite[Proof of theorem 6.9]{BalCaiFio:Poi04} we have a commutative diagram 
\begin{equation*}
\xymatrix{ 	H^{2r}_{\rig,c}(\scrW_k/K)\ar[rr]^{cosp}\ar[dr]_{[w_k]_\rig}& &\ar[ld]^{[w_K]_\dR} H^{2r}_{\dR,c}(\scrW_K/K)\\
	& K }
\end{equation*}
where $[w_k]_\rig$ (resp. $[w_K]_\dR$) is the rigid (resp. de Rham) trace map and it is an isomorphism of $K$-vector spaces. 

Given  a general relative cycle $w$  we can reduce to the basic case by arguing as follows.  First by linearity and the functoriality of the specialization map we can restrict to the case $w=\scrW$ with $\scrW$ integral. Then the generic fiber $\scrW_K$ is integral and  by \cite[Prop. 17.15.12]{EGA4.4} there exists a closed $K$-sub-scheme $T$ such that $\scrW_K\setminus T$ is smooth over $K$. Let $\mathscr{T}$ be the flat extension of $T$ (see Lemma~\ref{lmm:platification}), then $\mathscr{T}$ is of codimension $\ge 1$ in $\scrW$ and the complement $\scrW\setminus \mathscr{T}$ is a flat $\scrV$-scheme of relative dimension $r$. Consider the long exact sequence
\[
	\cdots H^{2r-1}_{\dR,c}(T)\to H^{2r}_{\dR,c}(\scrW_K\setminus\mathscr{T}_K)\to H^{2r}_{\dR,c}(\scrW_K)\to H^{2r}_{\dR,c}(T)\cdots .
\]
Note that here the first and last terms  vanish for dimensional reasons. The same happens (\textit{mutatis mutandis}) for the rigid cohomology of the special fiber. 

Hence from now on we can assume that $\scrW$ is integral and that its generic fiber $\scrW_K$ is smooth. In this setting we apply the Reduced Fibre Theorem for schemes \cite[Theorem 2.1']{BosLutRay:For95}: \textit{i.e.}, there exist a finite field extension $K'/K$  and a finite  morphism $f:\scrY \to \scrW\x_\scrV \scrV'$ of $\scrV'$-schemes such that
\begin{enumerate}[(i)]
	\item $f_{K'}:\scrY_{K'}\to \scrW_{K'}$ is an isomorphism;
	\item $\scrY/\scrV'$ is flat and has reduced geometric fibers.
\end{enumerate}

Recall that the co-specialization map commutes with finite field extension and the same holds for the rigid and the de Rham trace map. By \cite[Proof of Prop. 6.4]{Pet:Cla03}  the rigid fundamental/cycle class is preserved by finite morphisms, \textit{i.e.}, 
\[
	[\scrY_{k'}]_\rig\circ f_{k'}^*=[\scrW_{k'}]_\rig\ ,\quad f_{k'}^*:H^{2r}_{\rig,c}(\scrW_{k'})\to H^{2r}_{\rig,c}(\scrY_{k'})
\]
From the above discussion there  is no loss of generality in assuming  that $\scrW$ has reduced geometric fibers and smooth generic fiber $\scrW_K$. Now let $S$ be the singular locus of the special fiber $\scrW_k$. Again by \cite[Prop. 17.15.12]{EGA4.4} the complement $\scrW_k\setminus S$ is an open and dense sub-scheme of  $\scrW_k$ and it is smooth over $k$. The scheme $S$ has codimension $\ge 1$ in $\scrW_k$, hence $H^{2r}_{\rig,c}(\scrW_k\setminus S)\to H^{2r}_{\rig,c}(\scrW_k)$ is an isomorphism. From this follows that we can assume $\scrW$ to be smooth over $\scrV$ where we know that the claim is true.
\end{proof}

From now on assume $\scrX$ to be smooth over $\scrV$. By the  compatibility of (co)-specialization with Poincaré duality \cite[Theorem 6.9]{BalCaiFio:Poi04} the (de Rham or rigid) cycle class map defines an element $\eta_\rig(w_k)\in H^{2q}_{\rig,|w_k|}(\scrX_k/K)$ (resp. $\eta_\dR(w_K)\in H^{2q}_{\dR,|w_K|}(\scrX_K/K)$) compatible with respect to the specialization morphism 
\begin{equation}\label{eq:class compatibility}
	sp(\eta_\dR(w_K))=\eta_\rig(w_k) \ .
\end{equation}

Before stating the next corollary we need to introduce some further notation. Let $X$ be a smooth scheme over $K$ (resp. over $k$), then  the de Rham (resp. rigid) cycle class map factors throw the Chow groups inducing a map
	\[
		\eta_\dR:CH^q (X/K)\to H^{2q}_\dR(X) \qquad\ (\text{resp.}\	\eta_\rig:CH^q (X/k)\to H^{2q}_\rig(X/K))
	\] 
	where, by abuse of notation,   $\eta_\dR(W)$ (resp. $\eta_\rig(W)$) is  viewed as an element of  $H^{2q}_\dR(X)$  (resp. $H^{2q}_\rig(X/K)$) via the canonical map $H^{2q}_{\dR,W}(X)\to H^{2q}_{\dR}(X)$ (resp. $H^{2q}_{\rig,W}(X/K)\to H^{2q}_{\rig}(X/K)$), for any $W$ integral sub-scheme of codimension $q$ (See \cite[Prop.~7.8.1 p.60]{Har:On-75} and \cite[Cor.~7.6]{Pet:Cla03}). \\
Also we recall that in  \cite[Exp.~X Appendix]{SGA6} 
is constructed a specialization map for the classical Chow ring
\[
	sp_{CH}:CH^*(\scrX_K/K)\ox\Q\rightarrow CH^*(\scrX_k/k)\ox\Q
\] 
Explicitly the map is given as follows. Let $W\subset \scrX_K$ be an integral scheme of codimension $q$ and denote by $\scrW$ its Zariski closure in $\scrX$. Then the specialization of $sp_{CH}[W]$ is the class representing the sub-scheme $\scrW_k$.
\begin{crl}\label{crl:cycle class + sp}
	With the above notation the following diagram commutes
	\begin{equation*}
	\xymatrix{
	CH^q (\scrX_K/K)\ox\Q\ar[d]_{sp_{CH}}\ar[r]^{\eta_\dR}& H^{2q}_\dR(\scrX_K)\ar[d]^{sp}\\
	CH^q(\scrX_k/k)\ox\Q\ar[r]_{\eta_\rig} & H^{2q}_\rig(\scrX_k/K) }
	\end{equation*}
	(we tensored each term by $\Q$ to guarantee the existence of $sp_{CH}$). 
\end{crl}
\begin{proof}
	Straightforward.
\end{proof}
Let now $\scrX$ be a smooth $\scrV$-scheme.  Besser defined the (rigid) syntomic cohomology groups $H^n_\syn(\scrX,i)$ (\textit{cf.} \cite[Def~6.1]{Bes:Syn00}). We will be rather sketchy on the definitions because we will give another construction later.\\
For   such a cohomology   there is a long exact sequence
\begin{equation}\label{eq:les}
		\cdots\to H_\syn^n(\scrX,i)\to H^n_\rig(\scrX_k/K_0)\x F^iH^n_\dR(\scrX_K)\xrightarrow{h}H^n_\rig(\scrX_k/K_0)\x H^n_\rig(\scrX_k/K)\stackrel{+}{\rightarrow}\cdots
\end{equation}

where $h(x_0,x_\dR)=(\phi(x_0^\sigma)-p^ix_0,x_0\ox 1_K-sp(x_\dR))$. \\

Roughly these groups are defined as  $H_\syn^n(\scrX,i):=H^n(\R\Gamma_{{\rm Bes}}(\scrX,i))$, where
\[
	  \R\Gamma_{{\rm Bes}}(\scrX,i):=\Cone(\R\Gamma_\rig(\scrX/K_0)\oplus {Fil}^i\R\Gamma_\dR(\scrX)\to \R\Gamma_\rig(\scrX/K_0)\oplus \R\Gamma_\rig(\scrX/K))[-1]
\]
is a complex of abelian groups functorial in $\scrX$ and such that
\[
	H^n(\R\Gamma_\rig(\scrX/K_0))=H_\rig^n(\scrX_k/K_0)\ ,\quad H^n({Fil}^i\R\Gamma_\dR(\scrX))=H^n_\dR(\scrX_K)\ . 
\]
 The functoriality of $\R\Gamma_{{\rm Bes}}(-,i)$  allows us to give the following definition.
\begin{dfn}
	Let $\scrX$ be a smooth $\scrV$-scheme. Let $\scrZ\subset \scrX$ be a closed sub-scheme of $\scrX$. We define the syntomic complex with support in $\scrZ$ using the complexes defined by Besser  in the following way 
	\[
		\R\Gamma_{{\rm Bes},\scrZ}(\scrX,i):=\Cone(\R\Gamma_{\rm Bes}(\scrX,i)\to \R\Gamma_{\rm Bes}(\scrX\setminus \scrZ,i))[-1]
	\]
This is a complex of abelian groups functorial with respect to cartesian squares. This fact will be used in the proof of Prop.~\ref{prp:regsyn}.\\
The cohomology of this complex is  the \emph{syntomic cohomology with support} in $\scrZ$, denoted by 
	\[
		H^n_{\syn, \scrZ}(\scrX,i):=H^n(\R\Gamma_{{\rm Bes},\scrZ}(\scrX,i))
	\]
	so that we get a long exact sequence
	\[
		\cdots H^n_{\syn, \scrZ}(\scrX,i)\to H^n_{\syn}(\scrX,i)\to H^n_{\syn}(\scrX\setminus  \scrZ,i)\to \cdots\ .
	\]
\end{dfn}
\begin{prp}[Syntomic cycle class]\label{prp:syncyc}
	Let $\scrX$ be a smooth $\scrV$-scheme  let  $w\in z^q(\scrX/\scrV,0)$ be a relative cycle of codimension $q$.
	Let $\psi:H^{2q}_{\syn,|w|}(\scrX,q)\to H^{2q}_{\rig,|w_k|}(\scrX_k/K_0)\x F^{q} H^{2q}_{\dR,|w_K|}(\scrX_K/K)$ be the canonical map.
	 Then $\psi$ is injective and there exists a unique element $\eta_\syn(w)\in   H^{2q}_{\syn,|w|}(\scrX,q)$ such that  $\psi(\eta_\syn(w))=(\eta_\rig(w_k),\eta_\dR(w_K))$.
\end{prp}
\begin{proof}
	By the definition of syntomic cohomology with support there is a long exact sequence similar to \eqref{eq:les}
\[
		\cdots\to  H_{\rig,|w_k|}^{2q-1}(\scrX_k/K_0)\oplus H_{\rig,|w_k|}^{2q-1}(\scrX_k/K)\to  H^{2q}_{\syn,|w|}(X,q)\to 
	\]
			\[ H^{2q}_{\rig,|w_k|}(\scrX_k/K_0)\x F^{q} H^{2q}_{\dR,|w_K|}(\scrX_K/K)\to H_{\rig,|w_k|}^{2q}(\scrX_k/K_0)\oplus H_{\rig,|w_k|}^{2q}(\scrX_k/K)\to \cdots  
	\]
	The last term on the left vanishes because of weak purity in rigid cohomology  \cite[Cor.~5.7]{Ber:Fin97}. 
	It follows that $H^{2q}_{\syn,|w|}(\scrX,q)$ consists of the pairs $(x,y)\in H^{2q}_{\rig,|w_k|}(\scrX_k/K_0)\x F^{q} H^{2q}_{\dR,|w_K|}(\scrX_K/K)$ such that $\phi(x^\sigma)=p^qx$ and $sp(y)=x\ox 1_K$. By Hodge theory $F^{q} H^{2q}_{\dR,|w_K|}(\scrX_K/K)= H^{2q}_{\dR,|w_K|}(\scrX_K/K)$. Moreover the Frobenius acts on $[w_k]_\rig$ by multiplication for $p^q$ (see \cite[Prop. 7.13]{Pet:Cla03}). Hence we can easily conclude the proof by \eqref{eq:class compatibility}.
\end{proof}
\begin{lmm}[Functoriality]\label{lmm:functoriality}
	Let $f:\scrX'\to\scrX$ be a closed immersion of smooth $\scrV$-schemes  and let  $w\in z^q(\scrX/\scrV,0)$ be a relative cycle of codimension $q$. Assume that the pre-image $f\inv w$ lies in $z^q(\scrX'/\scrV,0)$, then  $f^*\eta_\syn(w)=\eta_\syn(f\inv w)$.
\end{lmm}
\begin{proof}
	It is not restrictive to assume that  $w=\scrW$ is an integral sub-scheme of $\scrX$ flat over $\scrV$.  
	
	We first show that is sufficient to prove that $f^*\eta_\dR(w_K)=\eta_\dR(f\inv w_K)$. In fact the syntomic cycle class can be viewed as an element $(\eta_\rig(w_k),\eta_\dR(w_K))$ in the direct sum of rigid and de Rham cohomology. Then note that  $sp(\eta_\dR(w_K))=\eta_\rig(w_k)\ox 1_K$ and that the specialization map is functorial. 
	
	To prove that $f^*\eta_\dR(\scrW_K)=\eta_\dR(f\inv \scrW_K)$ we first  reduce by excision to the case where  $\scrW_K$ is smooth over $K$ (just remove from $\scrW_K, \scrX_K$ and $\scrX_K'$  the singular points of $\scrW_K$).   In the same way we can further assume  $f\inv \scrW_K$ to be smooth over $K$. Now we can follow the same proof of \cite[Prop.~7.1]{Pet:Cla03} to conclude.
\end{proof}
\begin{lmm}[Homotopy]\label{lmm:hi}
The (rigid) syntomic cohomology is homotopy invariant, \textit{i.e.} \[
	H^n_\syn(\scrX\x \A^1_\scrV,q)\iso H^n_\syn(\scrX,q)\ .
\] 
\end{lmm}
\begin{proof}
	Just consider the long exact  sequences of syntomic cohomology and note that the de Rham cohomology (of smooth schemes)  is  homotopy invariant by \cite[Prop.~7.9.1]{Har:On-75}. The same holds for rigid cohomology, for instance using the  Künneth formula \cite{Ber:Dua97}.
\end{proof}
\begin{lmm}[Weak purity]
	Let $\scrX$ be a smooth $\scrV$-scheme. Let $\scrZ\subset \scrX$ be a closed sub-scheme of $\scrX$ of codimension $q$. Then 
	\[
		H^n_{\syn, \scrZ}(\scrX,i)=0\quad \text{for all} \ n< 2q\ .
	\]	
\end{lmm}
\begin{proof}
	This follows directly from the long exact sequence of syntomic cohomology and the weak purity in de Rham and rigid cohomology (\cite[\S7.2]{Har:On-75} and \cite[\S1]{Pet:Cla03}). 
\end{proof}
\begin{prp}\label{prp:regsyn}
	Let $\scrX$ be a smooth $\scrV$-scheme. The syntomic cycle class map induces a group homomorphism $\reg_\syn:CH^{i}(\scrX/\scrV,2i-n)\to H_\syn^{n}(\scrX,i)$.
\end{prp}
\begin{proof}
	The construction is analogous  to that provided in \cite{Blo:Alg86h}. Consider the cohomological double complex $\R\Gamma_{\rm Bes} (\Delta_\scrX^{-n},q)^m$, non-zero for $m\ge 0$, $n\le 0$; the differential in $n$ is induced by $\partial_i^n$ in the usual way. Similarly define the double complex 
	\[
		\R\Gamma_{{\rm Bes},\supp} (\Delta_\scrX^{-n},q)^m:=\colim_{w\in z^q(\scrX,-n)} \R\Gamma_{{\rm Bes},|w|} (\Delta_\scrX^{-n},q)^m\ .
	\]
	For technical reasons we truncate these complex (non-trivially) 
	\[
		A_{?}^{n,m}=\tau_{n\ge -N}\R\Gamma_{{\rm Bes},?} (\Delta_\scrX^{-n},q)^m \qquad ?=\vuoto,\ \supp
	\]
	 for $N$ even and $N>>0$. \\
	Consider the spectral sequence
	\[
		E_1^{n,m}:=H^m(A^{\bullet,n})\Rightarrow H^{n+m}(s A^{\bullet,*})
	\]
	where $s$ denotes the associated simple complex of a double complex. By homotopy invariance (Lemma~\ref{lmm:hi}) $E_1^{n,m}:=H_\syn^m(\Delta^{-n}_\scrX,q)$ is isomorphic to $H_\syn^m(\scrX,q)$ for $-N\le n\le 0$ and $m\ge 0$;  otherwise it is zero. Moreover $d_1^{n,m}=0$ except for $n$ even, $-N\le n<0$ and $m\ge 0$ in which case $d_1^{n,m}=\id$. This gives an isomorphism $H^i(sA^{n,m})\iso H_\syn^i(\scrX,q)$.\\
	In the spectral sequence 
	\[
		E_{1,\supp}^{n,m}:=H^m(A_{\supp}^{\bullet,n})\Rightarrow H^{n+m}(s A_{\supp}^{\bullet,*})
	\]
	we have
	\[
		E_{1,\supp}^{n,m}=\colim_{w\in z^q(X,-n)}H^m_{\syn,|w|}(\Delta_{\scrX}^{-n},q)\quad\text{for}\ -N\le n\le 0\ ,m\ge 0,
	\]
	and $E_{1,\supp}^{n,m}=0$ otherwise.  Applying lemma~\ref{lmm:functoriality} to the face morphisms it is easy to prove that the syntomic cycle class induces a natural map of complexes $\Gamma_{\scrX/\scrV}^\bullet(q)\to E_{1,\supp}^{\bullet-2q,2q}$ and hence for all $i$ a map 
	$CH^q(\scrX/\scrV,2q-i)\to E_{2,\supp}^{i-2q,2q}$. Due to weak purity the groups $E_{r,\supp}^{n,m}$ are zero for $m<2q$ and $r\ge 1$. Hence there are natural maps $E_{2,\supp}^{i-2q,2q}\to E_{\infty,\supp}^{i-2q,2q}\to H^{i}(s A_{\supp}^{\bullet,*})$. By construction there is map $H^{i}(s A_{\supp}^{\bullet,*})\to H^{i}(s A^{\bullet,*})$. Composing all these maps we obtain the expected map $CH^q(\scrX,2q-i)\to H^{i}_\syn (\scrX,q)$.
\end{proof}
\begin{crl}
	With the above notation there is a commutative diagram
	\begin{equation*}
	\xymatrix{
	CH^q(\scrX/\scrV)\ox\Q\ar[d]_{\reg_\syn} \ar[r]^{\gamma\ox 1_\Q} & CH^q(\scrX_K/K)\ox\Q\ar[d]_{\eta_\dR}\ar[r]^{sp_{CH}}& CH^q(\scrX_k/k)\ox\Q\ar[d]_{\eta_\rig}\\
	H^{2q}_\syn(\scrX,q)\ar[r]_\pi &F^qH^{2q}_\dR(\scrX_K)\ar[r]_{sp}& H^{2q}_\rig(\scrX_k/K)  
	}
	\end{equation*}
	where $\pi$ is the composition $H_\syn^n(\scrX,i)\to H^n_\rig(\scrX_k/K_0)\x F^iH^n_\dR(\scrX_K)\xrightarrow{{\rm pr}_2}F^iH^n_\dR(\scrX_K)$; $\gamma$ is the map described in the remark~\ref{rmk:gamma}.
\end{crl}
\begin{proof}
	Just note that in this case $\reg_\syn$ is the map induced by the syntomic cycle class in the usual way.
\end{proof}
%
%
%
%

%
%
\section{$p$-adic Hodge complexes}
Having defined a regulator map with values into the syntomic cohomology is tempting to check (some of) the Bloch-Ogus axioms for this theory. We address this problem by viewing the syntomic cohomology as an absolute cohomology theory.

Thus in this section we define a triangulated category of \emph{$p$-adic Hodge complexes} similar to  that of \cite{Ban:Syn02}.  See also \cite{Beu:Not86}, \cite{Hub:Mix95} and \cite[Ch. V \S2.3]{Lev:Mix98}.

The syntomic cohomology will be computed by Hom groups in this category.
\begin{dfn}[\textit{cf.} {\cite[2.1]{Ban:Syn02}}]
	Let $C_{\rig}^b(K)$ be the category of pairs $(M^\bullet,\phi)$, where
	\begin{enumerate}[(i)]
		\item $M^\bullet=M_0^\bullet\ox_{K_0}K$ where $M_0^\bullet$ is a complex in $C^b(K_0)$;
		\item (Frobenius structure) Let  $(M_0^\bullet)^\sigma:= M_0^\bullet\ox_\sigma K_0$. Then  $\phi:(M_0^\bullet)^\sigma\to M_0^\bullet$ is a $K_0$-linear morphism.
	\end{enumerate}

	The morphisms in this category are morphisms in $C^b(K_0)$ compatible with respect to the Frobenius structure. In this way we get an abelian category.
\end{dfn}
\begin{dfn}
	Let ${\rm Filt}_K$ be the category of $K$-vector spaces with a descending, exhaustive and separated filtration. We denote  $C_{\dR}^b(K)=C^b({\rm Filt}_K)$ and we write the objects of this category as pairs  $(M^\bullet,F)$, where
	\begin{enumerate}[(i)]
		\item $M^\bullet$ is a complex in $C^b(K)$;
		\item (Hodge filtration) $F$ is a (separated and exhaustive) filtration on $M^\bullet$.
	\end{enumerate} 
\end{dfn}
\begin{rmk}[Strictness]\label{rmk:strictness}
	We review some technical facts about filtered categories. For a full discussion see \cite[\S~2 and 3]{Hub:Mix95}.\\
	The category ${\rm Filt}_K$ (and also $C_{\dR}^b(K)$) is additive but not abelian. It is an exact category by taking for  short exact sequences that which are exact as sequences of $K$-vector spaces and such that the morphisms are strict w.r.t. the filtrations: recall that a morphism $f:(M,F)\to (N,F)$ is \emph{strict} if $f(F^iM)=F^iN\cap \Im(f)$. 
	
	An object $(M^\bullet, F)\in C^b_\dR(K)$ is a \emph{strict complex} if its differentials are strict as morphism of filtered vector spaces. Strict complexes can be characterized also by the fact that the canonical spectral sequence $H^q(F^p)\Rightarrow H^{p+q}(M^\bullet)$ degenerates at $E_1$.
	
	One can define canonical truncation functors on $C^b_\dR(K)$: for $M^\bullet\in C_{\dR}^b(K)$ let
	\[
		\tau_{\le n}(M^\bullet,F)^i:=\begin{cases}
			M^i & i<n\\
			\Ker(d^n)& i=n\\
			0& i>n
		\end{cases} \qquad 
		\tau_{\ge n}(M^\bullet,F)^i:=\begin{cases}
			 0 & i<n-1\\
			\Coim(d^n)& i=n-1\\
			M^i& i\ge n
		\end{cases}\ .
	\]
	It is important to note that for a strict complex $(M^\bullet,F)$ the naive cohomology object $\tau_{\le n}\tau_{\ge n}(M^\bullet ,F)$ agrees with the cohomology $H^n(M^\bullet)$ of the complex of $K$-vector spaces underlying $(M^\bullet,F)$ (see \cite[Prop.~2.1.3 and \S3]{Hub:Mix95}).
\end{rmk}
\begin{rmk}[Comparison/gluing functors]  
	 There is an exact functor $\Phi_\rig:C_{\rig}^b(K)\to C^b(K)$ (resp. $\Phi_\dR:C_{\dR}^b(K)\to C^b(K) $) induced by $(M^\bullet,\phi)\mapsto M^\bullet$ (resp. $(M^\bullet,F)\mapsto M^\bullet$).
\end{rmk}
\begin{dfn}[\textit{cf.} {\cite[Def.~2.2]{Ban:Syn02}}]\label{def:pHC}
 We define the  category $pHC$ of \emph{$p$-adic Hodge complexes} whose objects are systems $M=(M_\rig^\bullet,M_\dR^\bullet,M_K^\bullet,c,s)$, where
	\begin{enumerate}[(i)]
		\item $(M_\rig^\bullet,\phi)$ is an object of $C_{\rig}^b(K)$ and $H^*(M_\rig^\bullet)$ is finitely generated over $K$;
		\item $(M_\dR^\bullet,F)$ is an object of $C_{\dR}^b(K)$ and $H^*(M_\dR^\bullet)$ is finitely generated over $K$;
		\item $M_K^\bullet$ is an object of $C^b(K)$ and $c:M_\rig^\bullet\to M_K^\bullet$ (resp.  $s:M_\dR^\bullet\to M_K^\bullet$) is morphism in $C^b(K)$. Hence $c,s$  give a diagram in $C^b_K$
	\[
		M_\rig^\bullet \xrightarrow{c} M_K^\bullet \xleftarrow{s} M_\dR^\bullet\qquad .
	\]
	\end{enumerate}

	A morphism in $pHC$ is given by a system $f=(f_\rig,f_\dR,f_K)$ where $f_?:M_?^\bullet\to N_?^\bullet$ is a morphism in $C^b_\rig(K)$, $C^b_\dR(K)$, $C^b(K)$, for $?=\rig,\dR,K$ respectively, and such that they are compatible with respect to to the diagram in (iii) above. 
\end{dfn}
A \emph{homotopy} in $pHC$ is a system of homotopies $h_i$ compatible with the comparison maps $c,s$. We define the category $pHK$ to be the category $pHC$ modulo morphisms homotopic to zero. We say that a morphism $f=(f_\rig,f_\dR,f_K)$ in $pHC$ (or $pHK$) is a \emph{quasi-isomorphism} if:  $f_?$ is a quasi-isomorphism for $?=\rig,K$; $f_\dR$ is a filtered quasi-isomorphism, \textit{i.e.} $\gr_F(f_\dR)$ is a quasi-isomomorphism. Finally we say that $M\in pHC$ (or $pHK$) is \emph{acyclic} if $M_?=0$ is acyclic for any $?=\rig,\dR,K$.
\begin{lmm}\label{lmm:tstructure}
	\begin{enumerate}[(i)]
		\item The category $pHK$ is a triangulated category.
		\item The localization of $pHK$ with respect to the class of quasi-isomorphisms exists. This category, denoted by $pHD$, is a triangulated category.
		\item On the category $pHD$ it is possible to define a non-degenerate t-structure  (resp. truncation functors) compatible with the standard t-structure (resp. truncation functors)  defined on $C^b(K), C^b_\rig(K)$ and $C^b_\dR(K)$.
	\end{enumerate}
\end{lmm}
\begin{proof}
	The proof is the same as \cite[Prop.~2.6]{Ban:Syn02}. See also \cite[\S2]{Hub:Mix95} for a survey on how to derive  exact categories.
\end{proof}
%
%
\begin{rmk}
In the terminology of Huber \cite[\S4]{Hub:Mix95} the category $pHC$ is the \emph{glued exact category} of $C_\rig^b,C_\dR^b$ via $C_K^b$. The foregoing definition is inspired by \cite{Ban:Syn02} where Bannai constructs a \emph{rigid} glued exact category $C^b_{MF}$, \textit{i.e.}, the comparison maps are all quasi-isomorphisms. For our purposes we cannot impose this strong assumption. This is motivated by the fact that the (co)-specialization is not an isomorphism for a general smooth $\scrV$-scheme.
\end{rmk}
\begin{prg}[Notation]\label{prg:gammafunctor}
	Let $M^\bullet,{M'}^\bullet$ be two objects of $pHC$. Consider the following diagram ${\cal H}(M^\bullet,{M'}^\bullet)$ (of complexes of abelian groups)
	\begin{equation*}
	\xymatrix@C-50pt@R=19pt{
	\Hom^\bullet_{K_0}((M_0^\bullet)^\sigma,{M_0'}^\bullet)& &\Hom_K^\bullet(M_\rig^\bullet,{M_K'}^\bullet)& &\Hom_K^\bullet(M_\dR^\bullet,{M_K'}^\bullet)& \\
	 &\ar[lu]_{h_0}\Hom^\bullet_{K_0}(M_0^\bullet,{M_0'}^\bullet)\ar[ur]^{h_1}& &  \ar[lu]_{h_2}\Hom^\bullet_{K}(M_K^\bullet,{M_K'}^\bullet)\ar[ur]^{h_3} & & 
	\ar[lu]_{h_4}\Hom^{\bullet,F}_{K}(M_\dR^\bullet,{M_\dR'}^\bullet) 
	}
	\end{equation*}
	where  $h_0(x_0)=x_0\circ\phi-\phi'\circ x_0^\sigma$, $h_1(x_0)=c'\circ (x_0\ox\id_K)$, $h_2(x_K)=x_K \circ c$, $h_3(x_K)=x_K\circ s$, $h_4(x_\dR)=s'\circ x_\dR$.   Then define the complexes of abelian groups  $\Gamma_0(M^\bullet,{M'}^\bullet):=$direct sum of the bottom row, $\Gamma_1(M^\bullet,{M'}^\bullet):=$direct sum of the top row. Consider the morphism 
	\[
		\psi_{M^\bullet,{M'}^\bullet}:\Gamma_0(M^\bullet,{M'}^\bullet)\to \Gamma_1(M^\bullet,{M'}^\bullet)\quad (x_0,x_K,x_\dR)\mapsto (-h_0(x_0),h_1(x_0)-h_2(x_K),h_3(x_K)-h_4(x_\dR))
	\]
	Finally let $\Gamma(M^\bullet,{M'}^\bullet):= \Cone(\psi_{M^\bullet,{M'}^\bullet})[-1]$.
\end{prg}
\begin{rmk}\label{rmk:tatetwist}
	\begin{enumerate}[(i)]
		\item Let $\K(-n)$ be the Tate twist p-adic Hodge complex: $\K(-n)_\rig$ (resp. $\K(-n)_\dR, \K(-n)$) is equal to $K$ concentrated in degree zero; the Frobenius is $\phi(\lambda):=p^n\sigma(\lambda)$; the filtration is $F^i=K$ for $i\le n$ and zero otherwise.  
		\item Given two $p$-adic Hodge complexes $M^\bullet$ and ${M'}^\bullet$ we define their tensor product $M^\bullet\ox {M'}^\bullet$ component-wise, \textit{i.e.}, $(M_\rig^\bullet\ox {M_\rig'}^\bullet,M_\dR^\bullet\ox {M_\dR'}^\bullet,M_K^\bullet\ox {M_K'}^\bullet,c\ox c', s\ox s')$. We denote by $M^\bullet(n)$ the complex $M^\bullet\ox \K(n)$.
		\item The complex $\Gamma(\K,M^\bullet(n))$ is quasi-isomorphic to
		\[
			\Cone(M_0^\bullet\oplus F^n M_\dR^\bullet \xrightarrow{\eta} M_0^\bullet\oplus M_K^\bullet)[-1]\quad \eta(x_0,x_\dR)=(p^{-n}\phi(x_0^\sigma)-x_0,c(x_0\ox \id_K)-s(x_\dR))
		\]
		where $x_0\in M_0^\bullet,\ x_\dR\in F^nM_\dR^\bullet$.

		If $c$ is a quasi-isomorphism, letting $sp$ denote the composition of
		\[
			H^q(F^nM_\dR^\bullet) \stackrel{s^*}{\rightarrow} H^q(M_K^\bullet) \xleftarrow[\iso]{c^*} H^q(M_\rig^\bullet)\ ,
		\]
		we obtain a long exact sequence
		\[
		\cdots \to 	H^q(\Gamma(\K,M^\bullet(n)))\to H^q(M_0^\bullet)\oplus H^q(F^nM_\dR^\bullet)\stackrel{\eta'}{\rightarrow} H^q(M_0^\bullet)\oplus H^q(M_\rig^\bullet)\to \cdots,
		\]
		where $\eta'(x_0,x_\dR)=(p^{-n}\phi(x_0^\sigma)-x_0,x_0\ox 1_K- sp(x_\dR))$.

		If $s$ is a quasi-isomorphism, letting $cosp$ denote the composition of
		\[
			H^q(M_\rig^\bullet) \stackrel{c^*}{\rightarrow} H^q(M_K^\bullet) \xleftarrow[\iso]{s^*} H^q(M_\dR^\bullet)\ ,
		\]
		we obtain a long exact sequence
		\[
		\cdots \to 	H^q(\Gamma(\K,M^\bullet(n)))\to H^q(M_0^\bullet)\oplus H^q(F^nM_\dR^\bullet)\stackrel{\eta''}{\rightarrow} H^q(M_0^\bullet)\oplus H^q(M_\dR^\bullet)\to \cdots,
		\]
		where $\eta''(x_0,x_\dR)=(p^{-n}\phi(x_0^\sigma)-x_0,cosp(x_0\ox 1_K)- x_\dR)$.
	\end{enumerate}
\end{rmk}
\begin{prp}[Ext-formula]\label{prp:extformula}
	With the above notation there is a canonical morphism of abelian groups
	\[
		\Hom_{pHD}(M^\bullet,{M'}^\bullet[n])\iso H^n(\Gamma(M^\bullet,{M'}^\bullet))\ .
	\]
	In particular if $M^\bullet=M,{M'}^\bullet=M'$ are concentrated in degree $0$, $H^n(\Gamma(M,M'))=0$ for $n\ge 2$ and $n<0$.
\end{prp}
\begin{proof}
	By the octahedron axiom we have the following triangle in $D^b(Ab)$
	\[
		\Ker \psi_{M^\bullet,{M'}^\bullet}\to \Gamma(M^\bullet,{M'}^\bullet)\to \Coker \psi_{M^\bullet,{M'}^\bullet}[-1]\xrightarrow{+}\ .
	\] 
	Its cohomological long exact sequence is 
	\[
		\xrightarrow{\partial}H^n(\Ker \psi_{M^\bullet,M'})\to H^n(\Gamma(M,{M'}^\bullet))\to H^n(\Coker \psi_{M^\bullet,{M'}^\bullet}[-1])\xrightarrow{\partial}
	\]
	Note that by construction $H^n(\Ker \psi_{M^\bullet,{M'}^\bullet})=\Hom_{pHK}(M^\bullet,{M'}^\bullet[n])$. Let $I$ be the family of quasi-isomorphisms $g:{M'}^\bullet\to {M''}^\bullet$, then $\Hom_{pHD}(M^\bullet,{M'}^\bullet[n])=\colim_I\Hom_{pHK}(M^\bullet,{M''}^\bullet[n])$. Thus the result is proved if we show that: 
	
	i) $H^n(\Gamma(M^\bullet,{M'}^\bullet))\iso H^n(\Gamma(M^\bullet,{M''}^\bullet))$ for any $g:{M'}^\bullet\to {M''}^\bullet$ quasi-isomorphism.
	
	ii) $\colim_I H^n(\Coker \psi_{M^\bullet,{M''}^\bullet}[-1])=0$.
	
	The first claim follows from the exactness of $\Gamma(M^\bullet,-)$ and the second is proved  in \cite[1.7,1.8]{Beu:Not86} (and with more details in \cite[Lemma 4.2.8]{Hub:Mix95} or \cite[Lemma 2.15]{Ban:Syn02}) with the assumption that all the gluing maps are quasi-isomorphisms, but this hypothesis is not necessary. 
\end{proof}
\begin{lmm}[Tensor Product]\label{lmm:diag tensor}	Let $M^\bullet, {M'}^\bullet, I^\bullet$ be $p$-adic Hodge complexes. For any $\alpha \in K$ there is a morphism of complexes
\[
	\cup_\alpha:\Gamma (I^\bullet,M^\bullet)\ox \Gamma (I^\bullet,{M'}^\bullet)\to \Gamma (I^\bullet,M^\bullet\ox {M'}^\bullet)
\]
such that they are all  equivalent up to homotopy. 
\end{lmm}
\begin{proof}
	See \cite[1.11]{Beu:Not86}.
\end{proof}
\begin{rmk}[Enlarging the diagram]\label{rmk:enlarging}
	We recall some results from \cite[Ch. V, 2.3.3]{Lev:Mix98}. Let $M_1^\bullet\xrightarrow{f} M_2^\bullet \xleftarrow{g} M_3^\bullet$ (resp. $M_1^\bullet\xleftarrow{f} M_2^\bullet\xrightarrow{g} M_3^\bullet$) be a diagram of complexes in $C^b(K)$.  Let  $P^\bullet=\Cone(f-g:M_1^\bullet\x M_3^\bullet\to M_2^\bullet)[-1]$  be the \emph{quasi pull-back} complex (resp.  $Q^\bullet=\Cone((f,-g):M_2^\bullet\to M_1^\bullet\x M_3^\bullet)$ be the \emph{quasi push-out}). Assume that $f$ is quasi-isomorphism then we have the following diagrams commutative up to homotopy
	\begin{equation*}
	\xymatrix{
	P^\bullet \ar[d]_{}\ar[r]^{h}&    M^\bullet_3\ar[d]^{g}& & & M_2^{\bullet}\ar[d]_{g}\ar[r]^{f}&    M_1^\bullet\ar[d]^{}\\
	M_1^\bullet\ar[r]_{f} & M_2^\bullet & & &M_3^\bullet\ar[r]_{k} & P^\bullet }
	\end{equation*}
	such that $h$ and $k$ are quasi-isomorphisms.
	
	Now let $pHC'$ be a category of systems $(M_\rig^\bullet,M_\dR^\bullet,M_1^\bullet,M_2^\bullet,M_3^\bullet,c,s,f,g)$ similarly to Def.~\ref{def:pHC} and such that there is  a diagram
	\[
		M_\rig^\bullet \xrightarrow{c} M_1^\bullet \xleftarrow{f}M_2^\bullet \xrightarrow{g}M_3^\bullet \xleftarrow{s} M_\dR^\bullet
	\]
Then  the quasi push-out induces a functor from the category $pHC'$ to the category $pHC$. This functor is  compatible with tensor product after passing to the categories $pHK'$ and $pHK$. 
\end{rmk}
\section{Godement resolution}\label{sec:gdm}
Here we recall some facts about the \emph{generalized} Godement resolution, also called bar-resolution (we refer to \cite{Ivo:Rea05}, see also \cite[\S8.6]{Wei:An-94}).

Let $u:P\to X$ be a morphism of Grothendieck topologies  so that $P^{\sim}$ (resp. $X^{\sim}$) is the category of abelian sheaves on $P$ (resp. $X$). Then we have a 
 pair of adjoint functors $(u^*,u_*)$, where $u^*:X^{\sim} \to P^{\sim}$, $u_*:P^{\sim} \to X^{\sim}$.  For any object ${\cal F}$ of $X^{\sim}$
we can define a co-simplicial object $B^*({\cal F}):\Delta\to  X^{\sim}$  in the following way. First  let $\eta:\id_{X^{\sim}}\to u_*u^*$ and $\epsilon:u^*u_*\to \id_{P^{\sim}}$ be the natural transformations induced by adjunction.

Endow $B^{n+1}({\cal F}):=(u_*u^*)^n({\cal F})$ with co-degeneracy maps
\[
	\sigma_i^n:=(u_*u^*)^i u_*\epsilon u^* (u_*u^*)^{n-1-i}:B^{n+1}({\cal F})\to B^n({\cal F})\quad i=0,..., n-1
\]
and co-faces
\[
	\delta^{n-1}_i:=(u_*u^*)^i\eta(u_*u^*)^{n-i}:B^{n}({\cal F})\to B^{n+1}({\cal F})\quad i=0,...,n\ .
\]
\begin{lmm}
	With the above notation let $sB^*({\cal F})$ be the associated complex of objects of $X^{\sim}$. Then there is a canonical map $b_{\cal F}:{\cal F}\to sB^*({\cal F})$ such that $u^*(b_{\cal F})$ is a quasi-isomorphism. Moreover if $u^*$ is exact and conservative $b_{\cal F}$ is a quasi-isomorphism. 
\end{lmm}
\begin{proof}
	See \cite[Ch. III, Lemma 3.4.1]{Ivo:Rea05}.
\end{proof}
 Then for any sheaf ${\cal F}\in  X^{\sim}$ (or complex of sheaves) we can define a functorial map  $b_{\cal F}:{\cal F}\to sB^*({\cal F})$ with $sB^n({\cal F}):= (u_*u^*)^{n+1}{\cal F}$.  We will denote this complex of sheaves by  $\Gdm_P({\cal F})$. In the case $u^*$ is exact and conservative  $\Gdm_P({\cal F})$  is a canonical resolution of ${\cal F}$. If ${\cal F}^\bullet$ is a complex of sheaves on $X$ we denote by $\Gdm_P({\cal F}^\bullet)$ the simple complex $s(\Gdm_P({\cal F}^i)^j)$. Often we will need to iterate this process and we will write $\Gdm_P^2({\cal F}):=\Gdm_P(\Gdm_P({\cal F}))$.

Now suppose there is a commutative diagram of sites
\begin{equation*}
\xymatrix{
P \ar[d]_{u}\ar[r]^{g}&    Q\ar[d]^{v}\\
 X\ar[r]_{f} & Y  }
\end{equation*}
and a morphism of sheaves $a:{\cal G}\to f_*{\cal F}$ where ${\cal F}$ (resp. ${\cal G}$) is a sheaf on $X$ (resp. $Y$).
\begin{lmm}\label{lmm:gdmmorphism}
	There is a canonical map $\Gdm_Q ({\cal G})\to f_* \Gdm_P({\cal F})$ compatible with $b_{\cal F}$ and $b_{\cal G}$.
\end{lmm}
\begin{proof}
	We need only show that there is a canonical map $v_*v^* {\cal G}\to f_*u_*u^* {\cal F}$ lifting $a$. First consider the composition ${\cal G}\to f_* {\cal F}\to f_*u_*u^* {\cal F}$. Then we get a map ${\cal G}\to v_*g_* u^*{\cal F}$ because $v_*g_*=f_*u_*$. By adjunction this gives $v^*{\cal G} \to g_* u^*{\cal F} $. Then we apply $v_*$ and use $v_*g_*=f_*u_*$ to obtain the desired map.
\end{proof}
\begin{rmk}[Tensor product]
	The Godement resolution is compatible with tensor product, \textit{i.e.}, for any pair of sheaves ${\cal F},\cal G$ on $X$ there is a canonical quasi isomorphism $\Gdm_P({\cal F})\ox \Gdm_P({\cal G})\to \Gdm_P({\cal F}\ox {\cal G})$. The same holds for 
 complexes that are bounded below. (See \cite[Appendix A]{FriSus:The02})
\end{rmk}
\begin{dfn}[{\textit{cf.} \cite[Exp. IV, \S6]{SGA4.1}}]\label{def:points}
	Let $X$ be a site and $Sh(X)$ be the associated topos of sheaves of sets. A \emph{point} of $X$ is a morphism of topoi $\pi:Set \to Sh(X)$, \textit{i.e.}, a pair of adjoint functors $(\pi^*,\pi_*)$ such that $\pi^*$ is left exact.
\end{dfn}
\begin{exm}	Let $X$ be a scheme. Then any point $x$ of the topological space underling $X$ gives a point $\pi_x$ of the Zariski site of $X$. We call them Zariski points.\\
Let now $x$ be a geometric  point of $X$, then it induces a point $\pi_x$ for  the \'etale site  of $X$. We call them \'etale points of $X$.\\
Let ${\cal F}$ be a  Zariski (resp. \'etale) sheaf  $X$ and $P$ be the set of Zariski (resp. \'etale) points of $X$. Then the functor ${\cal F}\mapsto \sqcup_{\pi\in P} {\cal F}_\pi:=\pi^*{\cal F}$ is exact and conservative. In other words the Zariski (resp. \'etale) site of $X$ has enough points.
\end{exm}
\begin{dfn}[Points on rigid analytic spaces {\cite{PutSch:Poi95}}]\label{def:rigpt}
	Let $\cal X$ be a rigid analytic space over $K$. We recall that  a \emph{filter} $f$ on $\cal X$ is a collection $({\cal U}_\alpha)_\alpha$ of admissible open subsets of $X$ satisfying:
	\begin{enumerate}[(i)]
		\item ${\cal X}\in f$, $\vuoto\notin f$;
		\item if ${\cal U}_\alpha,{\cal U}_\beta \in f$ then ${\cal U}_\alpha \cap {\cal U}_\beta \in f$;
		\item if ${\cal U}_\alpha \in f$ and the admissible open ${\cal V}$ contains ${\cal U}_\alpha$, then ${\cal V}\in f$.
	\end{enumerate}
	A \emph{prime filter} on $\cal X$ is a filter $p$ satisfying moreover
	\begin{enumerate}[(iv)]
		\item if ${\cal U}\in p$ and ${\cal U}=\cup_{i\in I} {\cal U}_i'$ is an admissible covering of ${\cal U}$, then  ${\cal U}_{i_0}'\in p$ for some $i_0\in I$.
	\end{enumerate}
		
	Let $P(\cal X)$ be the set of all prime filters of $\cal X$. The filters on $\cal X$ are ordered with respect to inclusion. 
	We can give to $P(\cal X)$ a Grothendieck topology and define a morphism of sites $\sigma :P({\cal X})\to \cal X$. Also we denote by $Pt(\cal X)$ the set of prime filters with the discrete topology. Let $i:Pt({\cal X})\to P(X)$ be the canonical inclusion and $\xi=\sigma \circ i$. 
\end{dfn}
\begin{rmk}
	Let $p=({\cal U}_\alpha)_\alpha$ be a prime filter on $\cal X$ as above. Then $p$ is a point of the site $\cal X$ (see Def.~\ref{def:points}). Using the construction of the continuos map $\sigma$ of \cite{PutSch:Poi95} we get that the morphism of topoi $\pi: Set \to Sh(\cal X)$, associated to $p$, is defined in the following way: for any sheaf (of sets) ${\cal F}$ on $\cal X$ let $\pi^*({\cal F})=\colim_{\alpha}{\cal F}({\cal U}_\alpha)$; for any set $\cal S$ and ${\cal V}$ admissible open in $\cal X$, let 
	\[
		\pi_*{\cal S}({\cal V})=\begin{cases}
			{\cal S} &\text{if}\ {\cal V}= {\cal U}_\alpha\  \text{for some}\  \alpha\\
			0& \text{otherwise}
		\end{cases}
	\]
	where $0$ denotes the final object in the category $Set$. In fact with the above notations we get easily the adjunction
	\[
		\Hom_{Set}(\pi^*({\cal F}),{\cal S})=\lim_\alpha\Hom_{Set}({\cal F}({\cal U}_\alpha),{\cal S})=\Hom_{Sh({\cal X})}({\cal F},\pi_*{\cal S}) \ .
	\]
\end{rmk}
\begin{lmm}\label{lmm:vdpsch}
	With the above notation the functor $\xi\inv:Sh({\cal X})\to Sh(Pt({\cal X}))$ is exact and conservative. In other words for any $p\in Pt(\cal X)$ the functors $Sh({\cal X})\ni {\cal F}\mapsto {\cal F}_p$ are exact and ${\cal F}=0$ if all ${\cal F}_p=0$.
\end{lmm}
\begin{proof}
	See \cite[\S4]{PutSch:Poi95} after the proof of the Theorem 1.
\end{proof}
\section{Rigid and de Rham complexes}
We begin this section by  recalling the construction of the rigid complexes of  \cite[\S4]{Bes:Syn00}. Instead of the techniques of \cite[Vbis]{SGA4.2} we use the machinery of generalized Godement resolution as developed in Section~\ref{sec:gdm}.  
This alternative approach was also mentioned by Besser in the introduction of his paper.\\
 We then recall the construction of the de Rham complexes.\\[1ex]
We call a \emph{rigid triple} a system $(X,\bar X, \sf P)$ where: $X$ is an algebraic $k$-scheme; $j:X\to \bar X$ is an open embedding into a proper $k$-scheme; $\bar X\to \sf P$ is a closed embedding into a $p$-adic formal $\scrV$-scheme $\sf P$ which is smooth in a neighborhood of $X$.
\begin{dfn}[{\cite[4.2, 4.4]{Bes:Syn00}}]\label{def:compatible morph}
	Let $(X,\bar X, \sf P)$, $(Y,\bar Y,\sf Q)$ be two rigid triples and let $f:X\to Y$ be a morphism of $k$-schemes. Let ${\cal U}\subset ]\bar X[_P$ be a strict neighborhood of $]X[_{\sf P}$ and $F:{\cal U}\to {\sf Q}_K$ be a morphism of $K$-rigid spaces. We say that $F$ is \emph{compatible} with $f$ if it induces the following commutative diagram
	\begin{equation*}
	\xymatrix{
	]X[_{\sf P} \ar[d]_{sp}\ar[r]^{F}&   ]Y[_{\sf Q} \ar[d]^{sp}\\
	X\ar[r]_{f} &   Y}
	\end{equation*}
 We write $\Hom_f({\cal U},{\sf Q}_K)$ for the collection of morphism compatible with $f$.
\end{dfn}
The collection of rigid triples forms a category with the following morphisms: $\Hom((X,\bar X,{\sf P}),(Y,\bar Y, {\sf Q}))=$ the set of pairs $(f,F)$ where $f:X\to Y$ is a $k$-morphism and $F\in \colim_{\cal U}\Hom_f({\cal U},\sf Q)$. We denote it by $\rm RT$.

\begin{prp}\label{prp:rig}
	\begin{enumerate}[(i)]
		\item There is a functor from the category of algebraic $k$-schemes $\Sch/k$ 
		\[
			(\Sch/k)^\circ\rightarrow C(K_0)\quad X\mapsto R\Gamma_\rig(X/K_0)
		\]
		such that $H^i(R\Gamma_\rig(X/K_0))\iso H^i_\rig(X/K_0)$ .
		 Moreover there exists a canonical $\sigma$-linear endomorphism of $R\Gamma_{\rig}(X/K_0)$  inducing the Frobenius on cohomology.
		 \item There are two functors ${\rm RT}\rightarrow C(K)$
		\[
			 \widetilde{R\Gamma}_\rig(X)_{\bar X,{\sf P}} \ ,\quad R\Gamma_{\rig}(X/K)_{\bar X,\sf P}
		\]
		and functorial quasi-isomorphisms with respect to maps of rigid triples
		\[
			R\Gamma_{\rig}(X/K)\gets \widetilde{R\Gamma}_\rig(X)_{\bar X,\sf P}\to R\Gamma_\rig(X)_{\bar X,\sf P}\ .
		\]
	\end{enumerate}
\end{prp}
\begin{proof}
	See \cite[4.9, 4.21, 4.22]{Bes:Syn00}.
\end{proof}
\begin{rmk}\label{rmk:building block}
	The building block of the construction is the functor $R\Gamma_\rig(X/K)_{\bar X, \sf P}$. That complex  is constructed with  a system of  compatible resolution of the over-convergent de Rham complexes $j^\dag_X\Omega_{\cal U}^\bullet$ where $\cal U$ runs over all strict neighborhood of the tube of $X$. Using Godement resolution we can explicitly define
	\[
		R\Gamma_{\rig}(X/K)_{\bar X,\sf P}:=\colim_{\cal U}\Gamma({\cal U}, \Gdm_\an j^\dag_X\Gdm_\an \Omega_{\cal U}^\bullet)
	\] 
	where $\Gdm_\an=\Gdm_{Pt({\cal U})}$. This will be an essential ingredient for achieving the main results of the paper. \\
	All the proofs of \cite[\S4]{Bes:Syn00} work using this \emph{Godement} complex. We recall that 
	\[
		R\Gamma_{\rig}(X/K):=\colim_{A\in {\rm SET}^0_X}R\Gamma_{\rig}(X/K)_{\bar X_A,{\sf P}_A}\qquad \widetilde{R\Gamma}_{\rig}(X/K):=\colim_{A\in {\rm SET}^0_{(X,\bar X,{\sf P})}}R\Gamma_{\rig}(X/K)_{\bar X_A,{\sf P}_A}
	\]
	where ${\rm SET}^0_X$ and ${\rm SET}^0_{(X,\bar X,\sf P)}$ are filtered category of indexes.
\end{rmk}
With some modifications we can provide a compact support version of the above functors. We just need to be careful in the choice of morphisms of rigid triples. 
\begin{dfn}
	Let $(X,\bar X, \sf P)$, $(Y,\bar Y, \sf Q)$ be two rigid triples and let $f:X\to Y$ be a morphism of $k$-schemes. Let $F:{\cal U}\to {\sf Q}_K$ be compatible with $f$ (as in Def.~\ref{def:compatible morph}). We say that $F$ is \emph{strict} if 
	there is a commutative diagram

	$$\begin{CD}
	]X[_{\sf P} @>>>     {\cal U}                  @<<<              {\cal U}\setminus]X[_{\sf P}\\
	@VVFV                               @VVFV                               @VVFV  \\
	]Y[_{\sf Q} @>>>    {\cal V}                  @<<<              {\cal V}\setminus]Y[_{\sf Q}     \\
	\end{CD}$$
	where ${\cal V}$ is a strict neighborhood of $]Y[_{\sf Q}$ in $]\overline{Y}[_{\sf Q}$
\end{dfn}
It is easy to show that strict morphisms are composable. We denote by ${\rm RT}_c$ the category   of \emph{rigid triples with proper morphisms} which is the (not full) sub-category of $\rm RT$ with the same objects and morphisms pairs $(f,F)$, where $f$ is proper and $F$ is a germ of  a strict compatible morphism.
\begin{lmm}	\label{lmm:compactsupportrigid}
	\begin{enumerate}[(i)]
		\item  Let $(X,\bar X,\sf  P)$ be a rigid triple and let ${\cal U}$ be a strict neighborhood of $]X[_{\sf P}$. Then 
		\[
			H^i(\Gamma({\cal U},\Gdm_\an\underline{\Gamma}_{]X[_{\sf P}}\Gdm_\an \Omega_{\cal U}^\bullet))=H^i_{\rig,c}(X)\ .
		\]
		\item  Let $(Y,\bar Y,\sf Q)$ be another rigid triple, $f:X\to Y$ be a proper $k$-morphism and let $F:{\cal U}\to {\sf Q}_K$ be a morphism of $K$-analytic space compatible with $f$ and strict. Then there is a canonical map
		\[
			F^*:\Gdm_\an\underline{\Gamma}_{]Y[_{\sf Q}}\Gdm_\an \Omega_{{\sf Q}_K}^\bullet \rightarrow F_*\Gdm_\an\underline{\Gamma}_{]X[_{\sf P}}\Gdm_\an \Omega_{\cal U}^\bullet\ .
		\]
	\end{enumerate}
\end{lmm}
\begin{proof}
	i) It is sufficient to note that $\Gdm_\an (\Omega_{\cal U}^\bullet)$ is a complex of flasque sheaves and that a flasque sheaf is acyclic for $\underline{\Gamma}_{]X[_{\sf P}}$. Let $F$ be a flasque sheaf on the rigid analytic space ${\cal U}$. By definition $\underline{\Gamma}_{]X[_{\sf P}}=\Ker (a:F\to i_*i^*F)$. It is easy to check that $\R^qi_*i^*F=0$ for $\ge 1$. Hence $\R\underline{\Gamma}_{]X[_{\sf P}}F\iso \Cone(a:F\to i_*i^*F)[-1]$. But by hypothesis the map $a$ is surjective so that $\Cone(a:F\to i_*i^*F)[-1]\iso \Ker(a)=\R^0\underline{\Gamma}_{]X[_{\sf P}}$.
	
	ii) First consider the canonical pull-back of differential forms  $F^*:\Omega_{{\sf Q}_K}\to F_*\Omega_{\cal U}$. Then by  Lemma~\ref{lmm:gdmmorphism} we get a map $\Gdm_{Pt({\sf Q}_K)} \Omega_{{\sf Q}_K}^\bullet \rightarrow F_*\Gdm_{Pt({\cal U})} \Omega_{\cal U}$. Applying the functor $\underline{\Gamma}_{]X[_{\sf P}}$ to the adjoint map we get $\underline{\Gamma}_{]X[_{\sf P}}F\inv\Gdm_{Pt({\sf Q}_K)} \Omega_{{\sf Q}_K} \to \underline{\Gamma}_{]X[_{\sf P}}\Gdm_{Pt({\cal U})} \Omega_{\cal U}$. But the strictness of $F$ implies that there is a canonical map $F\inv \underline{\Gamma}_{]Y[_{\sf Q}}\to \underline{\Gamma}_{]X[_{\sf P}}F\inv$ (see \cite[Proof of Prop. 5.2.17]{Le-:Rig07}). Hence we have a map $F\inv\underline{\Gamma}_{]Y[_{\sf Q}}\Gdm_{Pt({\sf Q}_K)} \Omega_{{\sf Q}_K} \to \underline{\Gamma}_{]X[_{\sf P}}\Gdm_{Pt({\cal U})} \Omega_{\cal U}$. We can conclude the proof by taking the adjoint of this map and again applying Lemma~\ref{lmm:gdmmorphism}.
\end{proof}	
	\begin{prp}\label{prp:rig cpt}
		\begin{enumerate}[(i)]
			\item There is a functor from the category of algebraic  $k$-schemes with proper morphisms $\Sch_c/k$
			\[
		 (\Sch_c/k)^\circ\rightarrow C(K_0)\quad X\mapsto R\Gamma_{\rig,c}(X/K_0)
			\]
			such that  $H^i(R\Gamma_{\rig,c}(X/K_0))\iso H^i_{\rig,c}(X/K_0)$.
			 Moreover there exists a canonical $\sigma$-linear endomorphism of $R\Gamma_{\rig,c}(X/K_0)$  inducing the Frobenius on cohomology.
			 \item There are two functors ${\rm RT}_c\rightarrow C(K)$
			\[
				 \widetilde{R\Gamma}_{\rig,c}(X)_{\bar X,\sf P} \ ,\quad R\Gamma_{\rig,c}(X/K)_{\bar X,\sf P}
			\]
			and functorial quasi-isomorphisms with respect to maps of rigid triples
			\[
				R\Gamma_{\rig,c}(X/K)\gets \widetilde{R\Gamma}_{\rig,c}(X)_{\bar X,\sf P}\to R\Gamma_{\rig,c}(X)_{\bar X,\sf P}\ .
			\]
		\end{enumerate}
	\end{prp}
\begin{proof}
	In view of Lemma~\ref{lmm:compactsupportrigid} it sufficient to mimic the  construction given in  \cite[4.9, 4.21, 4.22]{Bes:Syn00} but using only proper morphisms of $k$-schemes and strict compatible maps. In this case the functors used in the construction are 
	
	\[
		R\Gamma_{\rig,c}(X/K)_{\bar X,\sf P}:=\colim_{\cal U}\Gamma({\cal U},\Gdm_\an\underline{\Gamma}_{]X[_{\sf P}}\Gdm_\an \Omega_{\cal U}^\bullet)
	\] 
and
	\[
		R\Gamma_{\rig,c}(X/K):=\colim_{A\in {\rm SET}^0_X}R\Gamma_{\rig,c}(X/K)_{\bar X_A,{\sf P}_A}\qquad \widetilde{R\Gamma}_{\rig,c}(X/K):=\colim_{A\in {\rm SET}^0_{(X,\bar X,{\sf P})}}R\Gamma_{\rig,c}(X/K)_{\bar X_A,{\sf P}_A}\ .
	\]
\end{proof}	
%
%
Now  we focus on de Rham complexes and we deal with smooth $K$-algebraic schemes. Let $X$ be a smooth algebraic $K$-scheme. The (algebraic) de Rham cohomology of $X$ is the hyper-cohomology of its complex of Kähler differentials $H^i_\dR(X/K):=H^i(X,\Omega_{X/K}^\bullet)$ (See \cite{Gro:On-66}). We can also define the de Rham cohomology with compact support (see \cite[\S1]{BalCaiFio:Poi04}) as the hyper-cohomology groups $H^i_{dR,c}(X/K):= H^i(\bar X,\lim_n J^n\Omega^\bullet_{\bar X/K})$ where $X\to \bar X$ is a smooth compactification and $J$ is the sheaf of ideals associated to the complement $\bar X\setminus X$ (this definition does not depend on the choice of $\bar X$ \cite[Thr. 1.8]{BalCaiFio:Poi04}). In order to consider the Hodge filtration on the de Rham cohomology groups we fix a normal crossings compactification $g:X\to Y$ and let $D:=Y\setminus X$ be the complement divisor (This is possible by the Nagata Compactification Theorem and the Hironaka Resolution Theorem, see \cite[\S3.2.1]{Del:TDH2}). We denote by $\Omega_Y^\bullet\langle D\rangle$ the  de Rham complex of $Y$ with logarithmic poles along $D$ (in the Zariski topology) (See \cite[3.3]{Jan:Mix90}). Let  $I\subset \O_Y$ be the defining sheaf of ideals of $D$. 
\begin{prp}\label{prp:derham}
	With the above notations
	\begin{enumerate}[(i)]
		\item there is a canonical isomorphism
		\[
			H^i_{\dR}(X/K)\iso H^i(Y,\Omega_Y^\bullet\langle D\rangle)\qquad (\text{resp.}\  H^i_{\dR,c}(X)\iso H^i(Y,I\Omega_Y^\bullet\langle D\rangle))\ .
		\]
		\item The spectral sequence 
		\[
			E_1^{p,q}=H^q(Y,\Omega^p)\Rightarrow H^i(Y,\Omega^\bullet) 
		\]
		degenerates at $1$ for $\Omega^\bullet=\Omega_Y^\bullet\langle D\rangle,\ I\Omega_Y^\bullet\langle D\rangle$.
		\item The filtration induced by the above spectral sequence on $H^i_{\dR}(X)$ (resp. $H^i_{\dR,c}(X)$) is independent by the choice of $Y$. Namely 
		\[
			F^jH^i_{\dR}(X):=H^i(Y,\sigma^{\ge j}\Omega_Y^\bullet\langle D\rangle)\qquad (\text{resp.}\  H^i_{\dR,c}(X):=H^i(Y,\sigma^{\ge j}I\Omega_Y^\bullet\langle D\rangle))
		\]
		where $\sigma^{\ge j}$ is the stupid filtration.
	\end{enumerate} 
\end{prp} 
\begin{proof}
	Using the  argument of \cite[3.2.11]{Del:TDH2} we get the  the independence of the choice of $Y$. The same holds for $H^i_{\dR,c}(X)$ 
	.\\
 Since our base field $K$ is  of characteristic $0$ we can find an embedding $\tau:K\to \C$. Then by GAGA we get an isomorphism of filtered vector spaces $H^i(Y,\Omega_Y^\bullet\langle D\rangle)\ox_K\C\iso H^i(Y_h,\Omega_{Y_h}^\bullet\langle D_h\rangle)$ (resp. $H^i(Y,I\Omega_Y^\bullet\langle D\rangle)\ox_K\C\iso H^i(Y_h,I_h\Omega_{Y_h}^\bullet\langle D_h\rangle)$), where $(-)_h$ is the complex analytification functor and $I_h$ is the defining sheaf of $D_h$. Thus we conclude by  \cite[\S3]{Del:TDH2} (resp. \cite[Part II, Ex. 7.25]{PetSte:Mix08}  for the compact support case).
\end{proof}
\begin{rmk}
	The degeneracy of the  spectral  sequence in (ii) of the above proposition can be proved algebraically (\cite{DelIll:Rel87}). We don't know an algebraic proof of the isomorphism in (i).
\end{rmk}
In the sequel a morphism of pairs $(X,Y)$ as above is a commutative square
\begin{equation*}
\xymatrix{
X \ar[d]_{u}\ar[r]^{g}&    Y\ar[d]^{v}\\
X'\ar[r]_{g'} & Y'  }
\end{equation*}
We say that the morphism is strict if the square is cartesian.\\
The complex $\Omega_Y^\bullet\langle D\rangle$ (resp. $I\Omega_Y^\bullet\langle D\rangle$) is a complex of Zariski sheaves over $Y$ functorial with respect to the pair $(X,Y)$ (resp. strict morphisms of pairs). We can construct two different (generalized) Godement resolutions (see \S~\ref{sec:gdm}): one using Zariski  points; the other via the $K$-analytic space associated to $Y$. \\
We will write $Pt(Y)=Pt(Y_\zar)$ for the set of Zariski points of $Y$ with the discrete topology; $Pt(Y_\an)$ for the discrete site of rigid points (Def.~\ref{def:rigpt}) of $Y_\an$; $Pt(Y_\an)\sqcup Pt(Y)$ is the direct sum in the category of sites.
\begin{prp}\label{prp:godementm derham comparison}
	With the notations of  Prop.~\ref{prp:derham} let $w: Y_\an\to Y_\zar$ be the canonical map from the rigid analytic site to the Zariski site of $Y$. Then for any Zariski sheaf $\Omega$ on $Y$ there is a diagram 
	\[
		\Gdm_{Pt(Y)}(\Omega )\gets \Gdm_{Pt(Y_\an)\sqcup Pt(Y)}(\Omega)\to w_*\Gdm_{Pt(Y_\an)}(w^*\Omega) 
	\]
	If we further consider $\Omega =\Omega_{Y}^\bullet\langle D \rangle$ (resp. $\Omega =I\Omega_{Y}^\bullet\langle D \rangle$), then  
the diagram is functorial with respect to the pair $(X,Y)$ (resp. $(X,Y)$ and strict morphisms). The same holds true with $\Gdm^2_?$ instead of $\Gdm_?$.
\end{prp}
\begin{proof}
	The first claim follows from Lemma~\ref{lmm:gdmmorphism} applied to the following commutative diagram of sites
	\begin{equation*}
	\xymatrix{
	Pt(Y_\an) \ar[d]_{}\ar[r]^{}&    Pt(Y_\an)\sqcup Pt(Y)\ar[d]^{}&\ar[l]Pt(Y)\ar[d]\\
	Y_\an\ar[r]_{w} & Y_\zar &\ar[l]^{\id} Y_\zar   }
	\end{equation*}
	with respect to the canonical map $\Omega\to w_*w^* \Omega$.
	
	The second claim follows from the functoriality of the complex $\Omega_{Y}^\bullet\langle D \rangle$ (resp. $I\Omega_{Y}^\bullet\langle D \rangle$).
\end{proof}
\begin{prp}
	\begin{enumerate}[(i)]
		\item Let $g:X\to Y$ be a normal crossing compactification as in Prop.~\ref{prp:derham}. Then there is a quasi isomorphism of  complexes of sheaves
		\[
		\Omega_Y^\bullet\langle D\rangle\to \Gdm^2_{Pt(Y)}(\Omega_Y^\bullet\langle D\rangle)\qquad  (\text{resp.}\  I\Omega_Y^\bullet\langle D\rangle\to \Gdm^2_{Pt(Y)}I\Omega_Y^\bullet\langle D\rangle )
		\]
		and the stupid filtration on $\Omega_Y^\bullet\langle D\rangle$  (resp. $I\Omega_Y^\bullet\langle D\rangle$) induces a filtration on the right term of the morphism.
		\item Let ${\rm Sm}/K$ (resp. ${\rm Sm}_c/K$) be the category of algebraic and smooth $K$-schemes  (resp. with proper morphisms). Let $D^b_\dR(K)$ the derived  category of the exact category of filtered vector spaces.
		Then there exist two   functors
		\[
			R\Gamma_\dR(-):({\rm Sm}/K)^\circ\rightarrow D^b_\dR(K) \qquad  R\Gamma_{\dR,c}(-):({\rm Sm}_c/K)^\circ\rightarrow D^b_\dR(K)
		\]
		such that $R\Gamma_\dR(X)=\Gamma(Y,\Gdm^2_{Pt(Y)}(\Omega_Y^\bullet\langle D\rangle))$, $R\Gamma_{\dR,c}(X)=\Gamma(Y,\Gdm^2_{Pt(Y)}I\Omega_Y^\bullet\langle D\rangle)$, with the same notation of (i).
		\item The filtered complexes $R\Gamma_\dR(X), \ R\Gamma_{\dR,c}(X)$ are strict (see \ref{rmk:strictness}).
	\end{enumerate}
\end{prp}
\begin{proof}
	i) This follows directly from the definition of Godement resolution.
	
	ii) This follows from the functoriality of the Godement resolution  with respect to morphism of pairs and 
\cite[3.2.11]{Del:TDH2} or 	 \cite[Lemma 15.2.3]{Hub:Mix95} for the compact support case.

	iii)  This follows by  \cite[Part II, \S\S4.3, 7.3.1]{PetSte:Mix08}.
\end{proof}
\section{Syntomic cohomology}
In this section we construct the $p$-adic Hodge complexes needed to define the rigid syntomic cohomology groups (also with compact support) for a smooth algebraic $\scrV$-scheme. The functoriality will be a direct consequence of the construction. 
\begin{lmm}\label{lmm:platification}
	Let $f:\scrX\to \scrV$ be a morphism of schemes and let $Z\subset \scrX_K$ be a closed sub-scheme 
of the generic fiber of $\scrX$. Then there exists a unique $\scrZ\subset \scrX$ closed sub-scheme which is flat over $\scrV$ and 
which satisfies $\scrZ_K=Z$. Thus, $\scrZ$ is the schematic closure of $Z$ in $\scrX$.
\end{lmm}
\begin{proof}
	See \cite[2.8.5]{Gro:Ele66III}.
\end{proof}
\begin{prp}\label{prp:gncd}
	Let $\scrX$ be a smooth  scheme over $\scrV$. Then there exists a \emph{generic normal crossings compactification}, \textit{i.e.}, an open embedding $g:\scrX\to \scrY$  such that:
	\begin{enumerate}[(i)]
		\item $\scrY$ is proper over $\scrV$;
		\item  $\scrY_K$ is smooth over $K$;
		\item $\scrD_K\subset \scrY_K$ is a normal crossings divisor, where  $\scrD=\scrY\setminus \scrX$.
	\end{enumerate} 
\end{prp}
\begin{proof}
	 First by Nagata (see \cite{Con:Del07}) there exists an open embedding $\scrX_K\to Y$, where  $Y$ is a proper $K$-scheme. By the Hironaka resolution theorem we can assume that  $Y$ is a smooth compactification of $\scrX_K$ with complement  a normal crossings divisor. Now we can define  $\scrY'$ to be the glueing of $\scrX$ and  $Y$  along the common open sub-scheme $\scrX_K$. These schemes are all of finite type over $\scrV$. It follows from the construction that $\scrY'$ is a scheme, separated  and of finite type over $\scrV$, whose generic fiber is $Y$.\\
	The Nagata compactification theorem works also in a relative setting, namely for a separated and finite type morphism: hence we can find a $\scrV$-scheme $\scrY$ which is a compactification of $\scrY'$ over $\scrV$. Then we get that  an open and dense embedding $h:\scrY'_K=Y\to \scrY_K$  of proper $K$-schemes. Hence $h$ is the identity and the statement is proven.
\end{proof}

\begin{rmk} We can also give another proof of the previous proposition assuming an embedding in a smooth $\scrV$-scheme.
	First by Nagata (see \cite{Con:Del07}) there exists an open embedding $\scrX\to \overline{\scrX}$. Now assume that $\overline{\scrX}$ is embeddable in a smooth $\scrV$-scheme $\scrW$ then by \cite[Th. 1.0.2]{Wod:Sim05}  we can get a resolution of the $K$-scheme $\overline{\scrX}_K$  by making a sequence of blow-up with respect to a family of closed subset in good position with respect to the regular locus of $\overline{\scrX}_K$, in particular  $Z_i\cap \scrX_K=\vuoto$. One can perform the same construction directly over $\scrV$, replacing the closed $Z_i\subset \scrW_K$ with their  Zariski closure $\scrZ_i$ in $\scrW$. By hypothesis  $Z_i\subset \scrW_K\setminus \scrX_K$, hence its closure $\scrZ_i$ is contained in $\scrW\setminus \scrX$ and  $(\scrZ_i)_K=Z_i$, by  Lemma~\ref{lmm:platification}. The construction doesn't affect what happens on the generic fiber because the blow-up construction is local and behaves well with respect to open immersions (See \cite[Ch.II 7.15]{Har:Alg77}). This will give $\scrY$ as in the proposition.
\end{rmk}

 From now on we keep the notation of Prop.~\ref{prp:gncd} with $g:\scrX\to \scrY$  being fixed. To simplify the notations we denote by ${\cal X}$ (resp. $\cal Y$) the rigid analytic space associated to the generic fiber of $\scrX$ (resp. $\scrY$), usually denoted by $\scrX_K^\an$.  Let $w:{\cal Y}\to (\scrY_K)_\zar$ be the canonical map of sites.\\
Sometimes we will simply write $\Gdm_\an=\Gdm_{Pt(\cal U)}$, if the rigid space $\cal U$ is clear from the context. Similarly we write $\Gdm_\zar=\Gdm_{Pt(X)}$ to denote the Godement resolution with respect to  Zariski points of a $K$-scheme $X$.
\begin{lmm}\label{lmm:connecting}
	With the above notations we have the following  morphisms of complexes of $K$-vector spaces
	\[
	a: \Gamma ({\cal Y},\Gdm_\an j^\dag \Gdm_\an  w^*\Omega_{\scrY_K}^\bullet\langle \scrD_K \rangle)\to \Gamma ({\cal X},\Gdm_\an j^\dag \Gdm_\an \Omega_{\cal X}^\bullet ) \to 	R\Gamma_\rig(\scrX_k/K)_{\scrY_k,\widehat{\scrY}} 
	\]
	and
	\[
	b:\Gamma ({\cal Y},\Gdm_\an \Gamma_{]\scrX_k[} \Gdm_\an  w^* I\Omega_{\scrY_K}^\bullet\langle \scrD_K \rangle)\to	\Gamma ({\cal X},\Gdm_\an \Gamma_{]\scrX_k[} \Gdm_\an \Omega_{\cal X}^\bullet ) \to  R\Gamma_{\rig,c}(\scrX_k /K)_{\scrY_k,\widehat{\scrY}}\ .
	\]
	All the maps are quasi-isomorphisms. We denote by $a$ (resp. $b$) the composition of the maps in the first (resp. second) diagram
\end{lmm}
\begin{proof}
	By construction (see Remark~\ref{rmk:building block} and Prop.~\ref{prp:rig cpt}) $R\Gamma_\rig (\scrX_k/K)_{\scrY_k,\widehat{\scrY}}$ (resp. $R\Gamma_{\rig,c}(\scrX_k /K)_{\scrY_k,\widehat{Y}}$) is a direct limit of complexes indexed over the strict neighborhoods of $]\scrX_k[_{\hat \scrY}$ and $\cal X$ is one of them. Hence the map on the right (of both diagrams) comes from the universal property of the direct limit.
	
	For the map on the left consider first the canonical inclusion of algebraic differential forms with log poles into the analytic ones $w^* \Omega_{\scrY_K}^\bullet\langle \scrD_K \rangle\to g^\an_* \Omega_{\cal X}^\bullet$. By Lemma~\ref{lmm:gdmmorphism} we get a map $\Gdm_{Pt({\cal Y})}w^* \Omega_{\scrY_K}^\bullet\langle \scrD_K \rangle\to g^\an_* \Gdm_{Pt({\cal X})} \Omega_{\cal X}^\bullet$. Then applying the $j^\dag$ functor and noting that $j^\dag g^\an_*=g_*^\an j^\dag$ (see \cite[5.1.14]{Le-:Rig07}) one obtains a 
morphism $j^\dag\Gdm_{Pt({\cal Y})}w^* \Omega_{\scrY_K}^\bullet\langle \scrD_K \rangle\to g^\an_*j^\dag  \Gdm_{Pt({\cal X})} \Omega_{\cal X}^\bullet$. This is what we need to apply  Lemma~\ref{lmm:gdmmorphism} again and conclude the proof for the first diagram. 
	
	For the second diagram one repeats the same argument using \cite[5.2.15]{Le-:Rig07}.
\end{proof}
\begin{lmm}\label{lmm:sp cosp}
	With the above notations we have the following  morphisms of complexes of $K$-vector spaces
	\[
	 a':\Gamma ({\cal Y},\Gdm_\an^2 w^*\Omega_{\scrY_K}^\bullet\langle \scrD_K \rangle) \to \Gamma ({\cal Y},\Gdm_\an j^\dag \Gdm_\an  w^*\Omega_{\scrY_K}^\bullet\langle \scrD_K \rangle) 
	\]
	and
	\[
		b': 	\Gamma ({\cal Y},\Gdm_\an \Gamma_{]\scrX_k[} \Gdm_\an  w^* I\Omega_{\scrY_K}^\bullet\langle \scrD_K \rangle)\to 	\Gamma ({\cal Y},\Gdm_\an^2 w^* I\Omega_{\scrY_K}^\bullet\langle \scrD_K \rangle)\ .
	\]
\end{lmm}
\begin{proof}
	The map $a'$ (resp. $b'$) is induced by the canonical map $\Omega\to j^\dag \Omega$  (resp. $\Gamma_{]\scrX_k[}\Omega\to \Omega$), where  $\Omega$ is  an abelian sheaf on $\cal Y$. In particular we consider $\Omega=\Gdm_\an w^*\Omega_{\scrY_K}^\bullet\langle \scrD_K \rangle)$ (resp. $\Omega=\Gdm_\an w^* I\Omega_{\scrY_K}^\bullet\langle \scrD_K \rangle)$). To conclude the proof apply the 
functor $\Gdm_\an$ again  and take global sections.
\end{proof}
\begin{prg}[Construction]\label{prg:construction}
	Now we put together all we have done  getting  a diagram, say $R\Gamma'(\scrX)$,  of complexes of $K$-vector spaces 
{\footnotesize	\begin{equation*}
	\xymatrix@C-68pt@R=15pt{
	&R\Gamma_\rig(\scrX/K)_{}\quad & &\quad {R\Gamma}_\rig(\scrX/K)_{\scrY_k,\hat \scrY}& &\Gamma({\cal Y},\Gdm_\an^2  w^*\Omega_{\scrY_K}\langle \scrD_K \rangle)& &\Gamma(\scrY_K,\Gdm_{\zar}^2  \Omega_{\scrY_K}\langle \scrD_K \rangle) \\
	\quad R\Gamma_\rig(\scrX/K_0)_{}\qquad \ar[ur]^{\alpha_1}& &\ar[lu]_{\alpha_2}  \quad \widetilde{R\Gamma}_\rig(\scrX/K)_{\scrY_k,\hat \scrY}\quad \ar[ur]^{\alpha_3} & &\ar[lu]_{\alpha_4}\Gamma({\cal Y},\Gdm_\an^2  w^*\Omega_{\scrY_K}\langle \scrD_K \rangle)\ar[ur]^{\alpha_5}& & \ar[lu]_{\alpha_6}\Gamma(\scrY_K,\Gdm_{\an +\zar}^2  \Omega_{\scrY_K}\langle \scrD_K \rangle)\ar[ru]^{\alpha_7}& & \ar[lu]_{\alpha_8}\Gamma(\scrY_K,\Gdm_{\zar}^2  \Omega_{\scrY_K}\langle \scrD_K \rangle)
	}
	\end{equation*}}
	where $\alpha_1,\alpha_5,\alpha_8$ are the identity maps; $\alpha_2,\alpha_3$ are the maps of Prop.~\ref{prp:rig}; 
$\alpha_4$ is the composition of $a\circ a'$  (see Lemma~\ref{lmm:connecting} and Lemma~\ref{lmm:sp cosp}); 
and $\alpha_6,\alpha_7$ are defined in Prop.~\ref{prp:godementm derham comparison}. By repeatedly applying the quasi-pushout 
construction we obtain a diagram of the following shape  
	\begin{equation}\label{eq:rgamma}
		R\Gamma_\rig(\scrX/K_0)\to R\Gamma_K(\scrX)\gets R\Gamma_\dR(\scrX)\ .
	\end{equation}
	It represents an object of $pHC$ which we denote by $R\Gamma(\scrX)$.

	Similarly we can construct a $p$-adic Hodge complex $R\Gamma_c(\scrX)$ associated to the diagram $R\Gamma_c'(\scrX)$ defined as follows
 {\footnotesize\begin{equation*}
	\xymatrix@C-69pt@R=15pt{
	&R\Gamma_{\rig,c}(\scrX/K)_{}\quad & &\quad {R\Gamma}_{\rig,c}(\scrX/K)_{\scrY_k,\hat \scrY}& &\Gamma({\cal Y},\Gdm_\an^2  w^* I\Omega_{\scrY_K}\langle \scrD_K \rangle)& &\Gamma(\scrY_K,\Gdm_{\zar}^2  I\Omega_{\scrY_K}\langle \scrD_K \rangle) \\
	\quad R\Gamma_{\rig,c}(\scrX/K_0)_{}\qquad \ar[ur]^{\beta_1}& &\ar[lu]_{\beta_2}  \quad \widetilde{R\Gamma}_{\rig,c}(\scrX/K)_{\scrY_k,\hat \scrY}\quad \ar[ur]^{\beta_3} & &\ar[lu]_{\beta_4}\Gamma({\cal Y},\Gdm_\an \Gamma_{]X[}\Gdm_\an  w^* I\Omega_{Y_K}\langle D \rangle)\ar[ur]^{\beta_5}& & \ar[lu]_{\beta_6}\Gamma(\scrY_K,\Gdm_{\an +\zar}^2  I\Omega_{\scrY_K}\langle \scrD_K \rangle)\ar[ru]^{\beta_7}& & \ar[lu]_{\beta_8}\Gamma(\scrY_K,\Gdm_{\zar}^2  I\Omega_{\scrY_K}\langle \scrD_K \rangle)
	}
	\end{equation*}
}	where $\beta_1,\beta_8$ are the identity maps; $\beta_2,\beta_3$ are the maps of Prop.~\ref{prp:rig}; $\beta_4$ is the map $b$  of Lemma~\ref{lmm:connecting};  $\beta_5=b'$ of Lemma~\ref{lmm:sp cosp}; and  $\beta_6,\beta_7$ are defined in Prop.~\ref{prp:godementm derham comparison}. Note that $\beta_6$ is a quasi-isomorphism by GAGA.
\end{prg}
%
%
\begin{prp}\label{prp:RGamma functors}
	Let $Sm/R$ (resp. $Sm_c/R$) be the category of algebraic and smooth $R$-schemes  (resp. with proper morphisms). 
	The previous construction induces the following  functors
	\[
		R\Gamma(-):({\rm Sm}/\scrV)^\circ\rightarrow pHD \qquad  R\Gamma_c(-):({\rm Sm}_c/\scrV)^\circ\rightarrow pHD
	\]
\end{prp}
\begin{proof}
	Let $f:\scrX\to \scrX'$ be a morphism of smooth $\scrV$-schemes. To get the functoriality we just have to show that can find two gncd compactifications $g:\scrX\to \scrY$ and $g':\scrX'\to \scrY'$ and a map $h:\scrY\to \scrY'$ extending $f$, \textit{i.e.}, $hg=g'f$. We argue as in \cite[\S3.2.11]{Del:TDH2}. Fix two gncd compactifications $g':\scrX'\to \scrY'$  and  $l:\scrX\to \scrZ$. Then consider the canonical map $\scrX\to \scrZ\x \scrY'$ induced by $l$ and $g'f$. Let $\bar \scrX$ be the closure of $\scrX$ in $\scrZ\x \scrY'$ and use the same argument  as in the proof of Prop.~\ref{prp:gncd} to get $\scrY$ which is generically a resolution of the singularities of ${\bar \scrX}$. Then by Prop.~\ref{prp:godementm derham comparison} and Prop.~\ref{prp:rig} we get the functoriality of $R\Gamma(-)$.
	
	If we further assume $f$ to be proper then we can apply the same argument in order to get the commutative square  $hg=g'f$ as above. Then by properness this square is also cartesian by \cite[Lemma 15.2.3]{Hub:Mix95}. From this fact and Prop.~\ref{prp:rig cpt}, \ref{prp:godementm derham comparison} we obtain the functoriality of $R\Gamma_c(-)$ with respect to proper maps.
\end{proof}
\begin{dfn}\label{def:syn}
	Let $\scrX$ be a smooth algebraic scheme over $\scrV$. For any $n,i$ integers we define the \emph{absolute cohomology}  groups of $\scrX$
	\begin{equation}\label{eq:syn}
		H_\abs^n(\scrX,i):=\Hom_{pHD}(\K,R\Gamma(\scrX)(i)[n])\quad (=H^n(\Gamma(\K,R\Gamma(\scrX)(i))))\ ;
	\end{equation}
the \emph{absolute cohomology with compact support} groups of $\scrX$
	\begin{equation}\label{eq:sync}
	H_{\abs,c}^n(\scrX,i):=\Hom_{pHD}(\K,R\Gamma_c(\scrX)(i)[n])\quad (=H^n(\Gamma(\K,R\Gamma_c(\scrX)(i))))\ .
	\end{equation}
\end{dfn}
A direct consequence of the definition is the existence of the following long exact sequence which should be considered to be 
a $p$-adic analog of the corresponding sequence for Deligne-Beilinson cohomology \cite[Introduction]{Beu:Hig84}.  
\begin{prp}
	With the above notation we have the following long exact sequences
	\[
		\cdots\to H_\abs^n(\scrX,i)\to H^n_\rig(\scrX_k/K_0)\x F^iH^n_\dR(\scrX_K)\xrightarrow{h}H^n_\rig(\scrX_k/K_0)\x H^n_\rig(\scrX_k/K)\stackrel{+}{\rightarrow}\cdots
	\]
	where $h(x_0,x_\dR)=(\phi(x_0^\sigma)-p^ix_0,x_0\ox 1_K-sp(x_\dR))$;
	\[
		\cdots\to H_{\abs,c}^n(\scrX,i)\to H^n_{\rig,c}(\scrX_k/K_0)\x F^iH^n_{\dR,c}(\scrX_K)\xrightarrow{h_c}H^n_{\rig,c}(\scrX_k/K_0)\x H^n_{\dR,c}(\scrX_K/K)\stackrel{+}{\rightarrow}\cdots
	\]
	where $h_c(x_0,x_\dR)=(\phi_c(x_0\sigma)-p^ix_0,cosp(x_0\ox 1_K)-x_\dR)$; 
	
\end{prp}
\begin{proof}
	By Prop.~\ref{prp:extformula} the absolute cohomology is the cohomology of a mapping cone, namely $\Gamma(\K,R\Gamma(\scrX))$ (or $\Gamma_c(\K,R\Gamma(\scrX))$ for the compact support case). The above long exact sequences are easily induced from the distinguished triangle defined by the term of the mapping cone (see Remark.~\ref{rmk:tatetwist} (iii)). Indeed, one need only note that the map $c:R\Gamma_\rig(\scrX/K_0)\ox K\to R\Gamma_K(\scrX)$ (resp. $s:R\Gamma_\dR(\scrX)\to R\Gamma_K(\scrX)$) is a quasi-isomorphism.
\end{proof}
\begin{prp}\label{prp:bessercomp}
	With the above notation there is a canonical isomorphism between the absolute cohomology we have defined and the (rigid) syntomic cohomology of Besser
	\[
		H_\syn^n(\scrX,i)\iso H_\abs^n(\scrX,i)\ .
	\]
\end{prp}
\begin{proof}
	Besser defines a complex  
	\[
	\R\Gamma_{\rm Bes}(\scrX,i):=\Cone(\R\Gamma_\rig(\scrX/K_0)\oplus {Fil}^i\R\Gamma_\dR(\scrX)\to \R\Gamma_\rig(\scrX/K_0)\oplus \R\Gamma_\rig(\scrX/K))[-1]
	\]
	and the syntomic cohomology groups (of degree $n$ and twisted $i$) by $H^n(\R\Gamma_{\rm Bes} (\scrX,i)$  (\textit{cf.} \cite[Proof of Prop. 6.3]{Bes:Syn00}). Note that, modulo the choice of the flasque resolution,  $\R\Gamma_\rig(\scrX/K_0)=R\Gamma_\rig(\scrX/K_0)$ (the left hand side is the notation used by Besser with the bold $\R$); the complex $Fil^i\R\Gamma_\dR(\scrX)$ is a direct limit over all the normal crossing compactifications of $\scrX_K$ of complexes of the form $\Gamma(Y,\Gdm^2_{Pt(Y)}\sigma^{\ge i}\Omega_{Y}\langle D\rangle)$, so that our $F^iR\Gamma_\dR (\scrX)$ is an element of this direct limit. To conclude the proof, 
recall that by  Remark~\ref{rmk:tatetwist} (iii) we obtain
	\[
		H_\abs^n(\scrX,i)\iso H^n(\Cone(R\Gamma_\rig(\scrX/K_0)\oplus F^iR\Gamma_\dR(\scrX)\to R\Gamma_\rig(\scrX/K_0)\oplus R\Gamma_\rig(\scrX/K))[-1])
	\]
	and it is easy to check that the maps in the two mapping cones are defined in the same way.
\end{proof}
\begin{rmk}
	The isomorphism $H^n(R\Gamma_\syn (\scrX,i))\iso \Hom_{pHD}(\K,R\Gamma(\scrX)(i)[n])$ can be viewed as a 
generalization of the result of Bannai \cite{Ban:Syn02}. He considers only smooth schemes $\scrX$ with a fixed compactification $\scrY$ 
and such that $\scrY\setminus \scrX=\scrD$ is a relative normal crossings divisor over $\scrV$. 
Moreover, Bannai's construction  is not functorial with respect to $\scrX$: functoriality holds only with respect 
to a so called syntomic datum. We should point out that  the category defined by Bannai  is endowed with a t-structure whose heart is the category $MF^f_K$ of (weakly) admissible filtered Frobenius modules. In our case we don't have such a nice picture.
\end{rmk}
\begin{rmk}\label{rmk:synhom}
	The category of $p$-adic Hodge complexes is not endowed with internal Hom. This is due to the fact that the Frobenius is only a quasi-isomorphism  hence we cannot invert it. In particular this happens for the complex $R\Gamma_{c}(\scrX)$: thus  we cannot define the dual $R\Gamma_c(\scrX)\dual:=\R\iHom(R\Gamma_c(\scrX),\K)$ and an absolute homology theory as
	\[
		H^{\abs}_n(\scrX,i):=\Hom_{pHD}(\K,R\Gamma_c(\scrX)\dual(-i)[-n]) \quad (\text{cf.}\ \cite[\S15.3]{Hub:Mix95}) .
	\] 
	Nevertheless the usual adjunction between Hom and $\ox$ should give a natural  isomorphism
	\[
	\Hom_{pHD}(\K,R\Gamma_c(\scrX)\dual)\iso 	\Hom_{pHD}(R\Gamma_c(\scrX),\K)
	\]
	and the right term does makes sense in our setting. This motivates the following definition.
\end{rmk}
\begin{dfn}
	With the notation of Def.~\ref{def:syn} we define the \emph{absolute homology}  groups of $\scrX$ as
	\begin{equation}\label{eq:synh}
	H^{\abs}_n(\scrX,i):=\Hom_{pHD}(R\Gamma_c(\scrX),\K(-i)[-n])\quad (=H^{-n}(\Gamma(R\Gamma_c(\scrX),\K(-i) ))   )\ .
	\end{equation}
\end{dfn}
\section{Cup product and Gysin map}
We are going to  prove that there is a morphism of $p$-adic Hodge complexes
\[
	R\Gamma(\scrX)\ox R\Gamma_c(\scrX)\to R\Gamma_c(\scrX)
\]
This induces a pairing at the level of the complexes computing absolute cohomology.
 The key point is the compatibility of the de Rham and rigid pairings with respect to the specialization and co-specialization maps.

Let us start by fixing some notation. Let  $\scrX$ be a smooth algebraic $\scrV$-scheme;   $g:\scrX\to \scrY$  be a gncd compactification; $\scrD=\scrY\setminus \scrX$ be the complement; $\cal Y$ is the rigid analytic space associated to the $K$-scheme $\scrY_K$; as before we denote by $w:{\cal Y}\to \scrY_{K,\zar}$ the canonical morphism of sites (see notations after Prop.~\ref{prp:gncd}). For any $\O_{\scrY_K}$-module $F$  we denote its pull-back by  $w^*$. 
\begin{rmk}	\label{rmk:dRpairing}
	\begin{enumerate}[(i)]
		\item The wedge product of algebraic differentials induces the following pairing
		\begin{equation*}
		{p_\dR}:\Omega_{\scrY_K}\langle \scrD_K\rangle\ox I\Omega_{\scrY_K}\langle \scrD_K\rangle\longrightarrow I\Omega_{\scrY_K}\langle \scrD_K\rangle  \ .
		\end{equation*}
		\item The analytification of  $p_\dR$ gives a pairing 
		\[
			w^*\Omega_{\scrY_K}\langle \scrD_K\rangle\ox w^* I\Omega_{\scrY_K}\langle \scrD_K\rangle \to w^*I\Omega_{\scrY_K}\langle \scrD_K\rangle
		\]
		Hence by \cite[Lemma 2.1]{Ber:Dua97} we get the following pairing
		\[
			p_\rig:j^\dag w^*\Omega_{\scrY_K}\langle \scrD_K\rangle\ox \underline{\Gamma}_{]\scrX_k[} w^*I\Omega_{\scrY_K}\langle \scrD_K\rangle \to \underline{\Gamma}_{]\scrX_k[} w^*I\Omega_{\scrY_K}\langle \scrD_K\rangle	\ .
		\]
	\end{enumerate}
\end{rmk}
\begin{lmm}[Sheaves level]
	The following diagram commutes
	{\small 
	\begin{equation*}
	\xymatrix@C50pt@R=20pt{
		{\Gdm_\an}^2 (w^* {\Omega_{\scrY_K}\langle \scrD_K\rangle})\ox \Gdm_\an^2 (w^*I\Omega_{\scrY_K}\langle \scrD_K\rangle)\ar[r]& \Gdm_\an^2 (w^*I\Omega_{\scrY_K}\langle \scrD_K\rangle)\\
		\Gdm_\an^2 (w^*\Omega_{\scrY_K}\langle \scrD_K\rangle)\ox \Gdm_\an\underline{\Gamma}_{]\scrX_k[}\Gdm_\an(w^*I\Omega_{\scrY_K}\langle \scrD_K\rangle)\ar[u]_{}\ar[d]\ar[r]^{\quad m}& \Gdm_\an\underline{\Gamma}_{]\scrX_k[}\Gdm_\an(w^*I\Omega_{\scrY_K}\langle \scrD_K\rangle)\ar[u]^{}    \ar[d]^{1}\\
\Gdm_\an j^\dag \Gdm_\an(w^*\Omega_{\scrY_K}\langle \scrD_K\rangle)\ox \Gdm_\an\underline{\Gamma}_{]\scrX_k[}\Gdm_\an (w^*I\Omega_{\scrY_K}\langle \scrD_K\rangle)	\ar[r] &  \Gdm_\an\underline{\Gamma}_{]\scrX_k[}\Gdm_\an(w^*I\Omega_{\scrY_K}\langle \scrD_K\rangle) 
}
	\end{equation*}}
	where $m:=p_\rig \circ (j^\dag\ox 1)$  by definition. 
\end{lmm} 
\begin{proof}
	The bottom square commutes by construction. Also  we get $p_\rig \circ (j^\dag\ox 1)=w^*p_\dR$ restricted to $w^*\Omega_{\scrY_K}\langle \scrD_K\rangle\ox \underline{\Gamma}_{]\scrX_k[}w^*I\Omega_{\scrY_K}\langle \scrD_K\rangle$. 
\end{proof}
\begin{prp}\label{prp:pairing}
	Let $\scrX$ be a smooth $\scrV$-scheme. Then there exists a morphism of $p$-adic Hodge complexes 
	\[
		\pi:R\Gamma(\scrX)\ox R\Gamma_c(\scrX)\to R\Gamma_c(\scrX)
	\]
	which is functorial with respect to $\scrX$ (as a morphism in $pHD$). Moreover taking the cohomology of this map we get the following compatibility
	\begin{equation*}
	\xymatrix{
	H^n_\dR(\scrX_K)\ox H^m_{\dR,c}(\scrX_K) \ar[r]^{}&   H^{n+m}_{\dR,c}(\scrX_K) \\
	H^n_\dR(\scrX_K)\ox H^m_{\rig,c}(\scrX_k)\ar[d]_{sp\ox\id}\ar[u]^{\id\ox cosp} \ar[r]_{} &   H^{n+m}_{\rig,c}(\scrX_k) \ar[d]_{\id}\ar[u]_{cosp}\\
	H^n_\rig(\scrX_k)\ox H^m_{\rig,c}(\scrX_k) \ar[r]^{}&   H^{n+m}_{\rig,c}(\scrX_k) }
	\end{equation*}
\end{prp}
\begin{proof}
	It is sufficient to provide a pairing of the enlarged diagrams (see Remark~\ref{rmk:enlarging}). Thus we have to define a morphism of diagrams $\pi':R\Gamma'(\scrX)\ox R\Gamma_c'(\scrX)\to R\Gamma_c'(\scrX)$ (notations as in \S~\ref{prg:construction}). Then it is easy to construct $\pi'$ using the previous lemma and the  compatibility of the Godement resolution with the tensor product.
\end{proof}
\begin{crl}\label{crl:pairing}
	There is a functorial pairing (induced by $\pi$ of the previous proposition)
	\[
		H^n_\abs(\scrX,i)\ox H^m_{\abs,c}(\scrX,j)\to H^{n+m}_{\abs,c}(\scrX,i+j)\ .
	\]
\end{crl}
\begin{proof}
	Consider the pairing of the Prop.~\ref{prp:pairing}. It induces a morphism 
 $R\Gamma(\scrX)(i)\ox R\Gamma_c(\scrX)(j)\to R\Gamma_c(\scrX)(i+j)$.  We then get the corollary  by Def.~\ref{def:syn} 
and Lemma~\ref{lmm:diag tensor}.
\end{proof}
\begin{prp}[Poincaré duality]\label{prp:dual}
	Let $\scrX$ be a smooth and algebraic $\scrV$-scheme of dimension $d$. Then there is a canonical isomorphism
	\[
		H^n_\abs(\scrX,i)\iso H_{2d-n}^\abs(\scrX,d-i)\ .
	\]
\end{prp}
\begin{proof}
	By definition (see equations \eqref{eq:syn}and \eqref{eq:synh}) it is sufficient to prove that the complex $\Gamma(\K,R\Gamma(\scrX))$ is quasi-isomorphic to the diagram $\Gamma(R\Gamma_c(\scrX),\K[-2d](-d))$.\\
	First recall  that $\Gamma(\K,R\Gamma(\scrX))$ is defined as
	\[
		\Cone(R\Gamma_\rig(\scrX/K_0)\oplus R\Gamma_K(\scrX)\oplus F^0 R\Gamma_\dR(\scrX) \stackrel{\psi}{\rightarrow} R\Gamma_\rig(\scrX/K_0)\oplus R\Gamma_K(\scrX)\oplus R\Gamma_K(\scrX))[-1]
	\]
	where $\psi(x_0,x_K,x_\dR)=(\phi(x_0)-x_0,c(x_0\ox \id_K)-x_K,x_K-s(x_\dR))$. In order to define the desired map we need to modify this complex replacing    $R\Gamma_K(\scrX)$ with $R\Gamma_\dR(\scrX)$ as follows 
	\[
			\Cone(R\Gamma_\rig(\scrX/K_0)\oplus R\Gamma_\dR(\scrX)\oplus F^0 R\Gamma_\dR(\scrX) \stackrel{\psi'}{\rightarrow} R\Gamma_\rig(\scrX/K_0)\oplus R\Gamma_K(\scrX)\oplus R\Gamma_\dR(\scrX))[-1]
	\]
	where $\psi'(x_0,x_\dR',x_\dR)=(\phi(x_0)-x_0,c(x_0\ox \id_K)-s(x_\dR'),x_\dR'-x_\dR)$. It is easy to see that this new complex, call it  $M^\bullet$, is quasi-isomorphic to   $\Gamma(\K,R\Gamma(\scrX))$.\\
	Recall that 	 the filtered complex $R\Gamma_{\dR,c}(\scrX)$ is strict, so that the   truncation $\tau_{\ge 2d}R\Gamma_{\dR,c}(\scrX)$ is the usual truncation of complexes of $K$-vector spaces (see lemma~\ref{lmm:tstructure}~(iii) and remark~\ref{rmk:strictness}). Then the cup product induces a morphism of complexes $M^\bullet\to \Gamma(R\Gamma_c(\scrX),\tau_{\ge 2d}R\Gamma_{c}(\scrX))$ which is a quasi-isomorphism
	 by the Poincaré duality theorems for rigid and de Rham cohomology (see \cite{Ber:Dua97}, \cite{Hub:Mix95}). Explicitly this map is induced by the following commutative diagram
{\footnotesize{	\begin{equation*}
	\xymatrix{
	 R\Gamma_\rig(\scrX/K_0)\oplus R\Gamma_\dR(\scrX)\oplus F^0 R\Gamma_\dR(\scrX)\ar[d]_{\psi'}\ar[r]^{\alpha\qquad\qquad\ 	  } & \Hom^\bullet_{K_0}(N_0,\tau_{\ge 2d}N_0)\oplus \Hom^\bullet_{K}(N_K,\tau_{\ge 2d}N_K) \oplus \Hom^{\bullet,F}_{K}(N_\dR,\tau_{\ge 2d}N_\dR) \ar[d]^{\xi}  \\
	R\Gamma_\rig(\scrX/K_0)\oplus R\Gamma_\rig(\scrX/K)\oplus R\Gamma_\dR(\scrX) 
		 \ar[r]_{\beta\qquad\qquad\ }   &  \Hom^\bullet_{K_0}(N_0^\sigma,\tau_{\ge 2d}N_0)		\oplus \Hom^\bullet_{K}(N_\rig,\tau_{\ge 2d}N_K)\oplus \Hom^\bullet_{K}(N_\dR,\tau_{\ge 2d}N_K)}
	\end{equation*}
	}
}	where $N:=R\Gamma_c(\scrX)$;
	\[
	\xi(f_0,f_K,\dR)=(\phi_c\circ f_0^\sigma-f_0\circ\phi_c,c\circ(f_0\ox\id_K)-f_K\circ c,f_K\circ s-s\circ x_\dR)\ ;	
	\]
	\[
	\alpha(x_0,x_\dR',x_\dR):(y_0,y_K,y_\dR)\mapsto (x_0\cup y_0, s(x_\dR')\cup y_K,x_\dR\cup y_\dR)\ ;
	\]
	\[
	\beta(x_0,x_\rig,x_\dR):(y_0,y_\rig,y_\dR)\mapsto (x_0\cup \phi_c( y_0),x_\rig\cup y_\rig, x_\dR\cup y_\dR)\ .
	\]

	To conclude the proof it is sufficient to apply the exact functor $\Gamma(R\Gamma_c(\scrX),-)$ to the following quasi-isomorphisms
	\[
		\tau_{\ge 2d}R\Gamma_{c}(\scrX)\gets H^{2d}(R\Gamma_c(\scrX))[-2d]\to \K(-d)[-2d]\ .
	\]
\end{proof}
\begin{rmk}
	We would like to point out some technical issues regarding  Poincaré duality in syntomic cohomology. 
	\begin{enumerate}[(i)]
		\item If it were possible to define an internal Hom in the category $pHC$  of $p$-adic Hodge complexes (see Remark~\ref{rmk:synhom}), then by Prop.~\ref{prp:pairing}  one would obtain the following natural isomorphism 
		\[
			R\Gamma(\scrX)(d)[2d]\iso R\Gamma_c(\scrX)\dual
		\]
		in the triangulated category $pHD$, where $R\Gamma_c(\scrX)\dual:=\R\iHom_{pHC}(R\Gamma_c(\scrX),\K)$ is the dual of $R\Gamma_c(\scrX)$. \\
		Then one would get the duality by adjunction
		\[
			\Hom_{pHD}(\K,R\Gamma(X)(i)[n])\iso \Hom_{pHD}(\K,R\Gamma_c(X)\dual(i-d)[n-2d])\iso \Hom_{pHD}(R\Gamma_c(X)(d-i)[2d-n],\K)\ .
		\]
		\item The Grothendieck-Leray spectral sequence for absolute homology is
				\[
					E_2^{p,q}=\Hom_{pHD}(H^{-q}(R\Gamma_c(\scrX)(i))[-p],\K),
				\]
			and it degenerates to the short exact sequences
				\[
					0\to \Hom_{pHD}(H^{n}(R\Gamma_c(\scrX)(i)),\K)\to H^{\syn}_n(\scrX,i) \to \Hom_{pHD}(H^{n+1}(R\Gamma_c(\scrX)(i)),\K[1])\to 0\ .
				\]

				It follows directly from Prop.~\ref{prp:extformula} that the  group $\Hom_{pHD}(H^{n}(R\Gamma_c(\scrX)(i)),\K)$ is 
				\[
				\{(x_0,x_\dR)\in H^n_{\rig,c}(\scrX_k/K_0)\dual\x (H^n_{\dR,c}(\scrX_K)/F^{i+1})\dual\ : (x_0\ox 1_K)=x_\dR\circ cosp,\ x_0\circ \phi_c=p^{i}x_0^\sigma\}\ .
				\]
			In cohomology the Frobenius is an invertible, hence the equation $x_0\circ \phi_c=p^{i}x_0^\sigma$ is equivalent to
			$x_0=p^ix_0^\sigma\phi_c\inv$ where the latter is the formula for the Frobenius of internal  $\iHom(H^n_{\rig,c}(\scrX_k),K(-i))$ in the category of isocrystals (\textit{i.e.}, modules with Frobenius). Hence by Poincaré duality for rigid cohomology we get $\phi((x_0\dual)^\sigma)=p^{d-i}x_0\dual$, $x\dual\in H_\rig^{2d-n}(\scrX_k/K_0)$ is the cohomology class corresponding to $x_0$, where $\phi$ is the Frobenius of $H_\rig^{2d-n}(\scrX_k/K_0)$.
			\item The  absolute homology is defined via the complex $R\Gamma_c(\scrX)$, but it is not clear how to relate it to the dual of the absolute cohomology with compact support.
	\end{enumerate}
\end{rmk}
\begin{crl}[Gysin map]\label{crl:gysin}
	Let $f:\scrX\to \scrY$ be a proper morphism of smooth algebraic $\scrV$-schemes of relative dimension  $d$ and $e$, respectively. Then there is a canonical map
	\[
		f_*:H_\abs^n(\scrX,i)\rightarrow H_\abs^{n+2c}(\scrY,i+c)
	\]
	where $c=e-d$.
\end{crl}
\begin{proof}
	This is a direct consequence of the previous proposition and the functoriality of $R\Gamma_c(\scrX)$ with respect to proper morphisms.
\end{proof}
\begin{rmk}
	With the above notations we get a morphism of spectral sequences 
	\[
		g:E_2^{p,q}(\scrX):=\Hom_{pHD}(\K,H^q(\scrX)(i)[p])\to E_2^{p,q+2c}(\scrY):=\Hom_{pHD}(\K,H^{q+2c}(\scrY)(i)[p])
	\] compatible with $f_*$. The map $g$ is induced by the Gysin morphism in de Rham and rigid cohomology.
\end{rmk}
	\addcontentsline{toc}{section}{References}
	

\end{document}